\documentclass[12pt]{paper}

\usepackage{color}
\usepackage{mathptmx}       % selects Times Roman as basic font
\usepackage{helvet}         % selects\textbf{\textbf{•}} Helvetica as sans-serif font
\usepackage{courier}        % selects Courier as typewriter font
\usepackage{type1cm}        % activate if the above 3 fonts are
\usepackage{makeidx}         % allows index generation
\usepackage{graphicx}        % standard LaTeX graphics tool
\usepackage{multicol}        % used for the two-column index
\usepackage[bottom]{footmisc}% places footnotes at page bottom
\usepackage{amsmath}
\usepackage{amsfonts}
\usepackage{amssymb}
\usepackage{amsthm}
\usepackage{caption}
\usepackage{subcaption}
\usepackage{comment}
\usepackage{amsopn}
\newtheorem{theorem}{Theorem}

\newtheorem{lemma}[theorem]{Lemma}

\newtheorem{proposition}[theorem]{Proposition}

\newtheorem{result}[theorem]{Proposition}
\usepackage[top=1in, bottom=1in, left=1in, right=1in]{geometry}
\numberwithin{theorem}{section}
\numberwithin{figure}{section}
\numberwithin{equation}{section}
\DeclareMathOperator{\cond}{\,|\,}
\DeclareMathOperator{\dist}{dist}
\DeclareMathOperator{\CR}{CR}
\DeclareMathOperator{\capa}{cap}
\DeclareMathOperator{\CM}{CM}
\DeclareMathOperator{\SLE}{SLE}
\DeclareMathOperator{\CLE}{CLE}

\DeclareMathOperator{\inte}{int}
\begin{document}

\title{Simple CLE in Doubly Connected Domains}
\author{Scott Sheffield, Samuel S. Watson and Hao Wu}

% your contribution title if the original one is too long%

\maketitle

\abstract{We study Conformal Loop Ensemble ($\CLE_{\kappa}$) in doubly connected domains: annuli, the punctured disc, and the punctured plane. We restrict attention to $\CLE_{\kappa}$ for which the loops are simple, i.e. $\kappa\in (8/3,4]$.
In \cite{SheffieldWernerCLE}, simple $\CLE$ in the unit disc is introduced and constructed as the collection of outer boundaries of outermost clusters of the Brownian loop soup. For simple $\CLE$ in the unit disc, any fixed interior point is almost surely surrounded by some loop of $\CLE$. The gasket of the collection of loops in $\CLE$, i.e. the set of points that are not surrounded by any loop, almost surely has Lebesgue measure zero. In the current paper, simple $\CLE$ in an annulus is constructed similarly: it is the collection of outer boundaries of outermost clusters of the Brownian loop soup conditioned on the event that there is no cluster disconnecting the two components of the boundary of the annulus. Simple $\CLE$ in the punctured disc can be viewed as simple $\CLE$ in the unit disc conditioned on the event that the origin is in the gasket. Simple $\CLE$ in the punctured plane can be viewed as simple $\CLE$ in the whole plane conditioned on the event that both the origin and infinity are in the gasket. We construct and study these three kinds of $\CLE$s, along with the corresponding exploration processes.
\smallbreak

}

\newcommand{\eps}{\epsilon}
\newcommand{\ov}{\overline}
\newcommand{\U}{\mathbb{U}}
\newcommand{\T}{\mathbb{T}}
\newcommand{\HH}{\mathbb{H}}
\newcommand{\LA}{\mathcal{A}}
\newcommand{\LD}{\mathcal{D}}
\newcommand{\LF}{\mathcal{F}}
\newcommand{\LK}{\mathcal{K}}
\newcommand{\LE}{\mathcal{E}}
\newcommand{\LL}{\mathcal{L}}
\newcommand{\LR}{\mathcal{R}}
\newcommand{\LU}{\mathcal{U}}
\newcommand{\LV}{\mathcal{V}}
\newcommand{\A}{\mathbb{A}}
\newcommand{\R}{\mathbb{R}}
\newcommand{\C}{\mathbb{C}}
\newcommand{\D}{\mathbb{D}}
\newcommand{\DO}{\mathbb{D}^{\dag}}
\newcommand{\CO}{\mathbb{C}^{\dag}}
\newcommand{\N}{\mathbb{N}}
\newcommand{\Z}{\mathbb{Z}}
\newcommand{\E}{\mathbb{E}}
\newcommand{\PP}{\mathbb{P}}
\newcommand{\QQ}{\mathbb{Q}}
\newcommand{\MR}{MR}

%\bigbreak
\noindent
%\textbf{Keywords:}
\section{Introduction}
\label{sec::introduction}
Schramm Loewner Evolution ($\SLE$) curves were introduced by Oded Schramm \cite{SchrammScalinglimitsLERWUST} as candidates for the scaling limit of various interfaces in discrete statistical physics models. For each $\kappa\ge 0$, $\SLE_{\kappa}$ is a random curve in a simply connected domain (which is non-empty and is not the whole plane) connecting one boundary point to another boundary point that satisfies certain conformal symmetry and so-called domain Markov property \cite{SchrammScalinglimitsLERWUST}. Since their introduction, $\SLE_{\kappa}$ have been proved to be the scaling limits of many discrete models. For example, $\SLE_3$ has been proved to be the scaling limit of the interface in critical Ising model \cite{ChelkakSmirnovIsing, ChelkakSmirnovHonglerIsing}; $\SLE_4$ has been proved to be the scaling limit of a level line of the discrete Gaussian Free Field \cite{SchrammSheffieldDiscreteGFF, SchrammSheffieldContinuumGFF}.

When one studies the scaling limit of the collection of all interfaces in a discrete statistical physics models (as opposed to a single interface), one is led to the notion of Conformal Loop Ensemble ($\CLE$). For each $\kappa\in (8/3,8]$, one can define $\CLE_{\kappa}$ in the unit disc which is a random countable collection of loops that are contained in the unit disc. Only for $\kappa\in (8/3,4]$, the loops are simple and disjoint. We occasionally use the term ``simple $\CLE$" to refer to a non-nested disjoint conformal loop ensemble $\CLE_{\kappa}$ for $\kappa\in (8/3,4]$, and we will focus exclusively on these $\CLE$s for $\kappa\in (8/3,4]$. In \cite{SheffieldExplorationTree, SheffieldWernerCLE}, simple $\CLE$ in the unit disc is defined and studied. The Brownian loop soup is a random collection of the Brownian loops which are Brownian paths start and end at the same point (see Section \ref{subsec::preliminaries_brownianloopsoup}). In \cite{SheffieldWernerCLE}, simple $\CLE$ in the unit disc is constructed from Brownian loop soup and the authors prove that $\CLE_{\kappa}$ for $\kappa\in (8/3,4]$ is the only one-parameter family of collections of loops that satisfies conformal invariance and the domain Markov property (as we will define in Section \ref{subsec::simplecle_def}), and each loop of which looks locally like an $\SLE_{\kappa}$. Now $\CLE_3$ is conjectured to be the scaling limit of the collection of interfaces in the critical Ising model; $\CLE_4$ has been proved to be the collection of level lines of Gaussian Free Field. (The details have not all written, but a reasonably detailed proof appears in Jason Miller's lecture slides \cite{MillerSheffieldCLE4}). Later in \cite{KemppainenWernerNestedSimpleCLERiemannSphere}, the nested $\CLE$ in Riemann sphere is defined and studied. Most of the effort in \cite{KemppainenWernerNestedSimpleCLERiemannSphere} is devoted to showing that the nested $\CLE$ in whole plane is invariant under inversion $z\mapsto 1/z$.

Given a collection of $\CLE$ loops in the unit disc, it is natural to ask what is the ``distance" between a loop, say the loop containing the origin denoted by $\gamma(0)$, and the boundary $\partial\U$, or what is the ``distance" between two loops. Since $\CLE$ is conformal invariant, such a distance should also be conformal invariant, and should depend on the collection of the loops between $\gamma(0)$ and the boundary. It turns out that the collection of loops between $\gamma(0)$ and $\partial \U$ is a collection of loops in the annulus. Therefore, to find such a distance between loops, we need to understand the properties of $\CLE$ in the annulus. This is the motivation for this paper. 

We construct $\CLE$ in the annulus as the collection of the outer boundaries of outermost clusters of Brownian loop soup in the annulus conditioned on the event that there is no cluster disconnecting the two components of the boundary of the annulus. Our main results about $\CLE$ in the annulus can be summarized as follows:
\begin{itemize}
\item $\CLE$ in the annulus satisfies an annulus version of conformal invariance and the domain Markov property (detailed description in Section \ref{sec::cle_annulus}).
\item $\CLE$ in the annulus and $\CLE$ in the unit disc are related in the following way: for a $\CLE$ in the unit disc, fix the loop containing a particular interior point. Then, given this loop, the conditional law of the collection of loops between this particular loop and the boundary of the domain has the same law as $\CLE$ in the annulus.
\end{itemize}

Consider $\CLE$ in the annulus with inradius $r\in (0,1)$ and outradius 1. We show that, as $r$ goes to zero, $\CLE$ in the annulus converges, and the limit object can be viewed as $\CLE$ in the unit disc ``conditioned" on the event that the origin is in the gasket. This is a variant of $\CLE$ in which the origin plays a special role. We call it $\CLE$ in the punctured disc. This version of $\CLE$ has the nice properties as we would expect:
\begin{itemize}
\item $\CLE$ in the punctured disc satisfies conformal invariance and the domain Markov property (see Section \ref{sec::cle_punctured_disc}).
\item The law of the set of loops that are ``near to the boundary of the disc" is approximately the same for $\CLE$ in the unit disc and $\CLE$ in the punctured disc (in a sense we will make precise in Proposition \ref{prop::cle_punctured_disc_simply_connected}).
\end{itemize}

In the construction of $\CLE$ in the punctured disc, we let the inradius of the annulus go to zero. We can also let the outradius go to infinity at the same time as the inradius goes to zero: Consider $\CLE$ in the annulus with inradius $r$ and outradius $1/r$. When $r$ goes to zero, $\CLE$ in the annulus also converges, and we call the limit object $\CLE$ in the punctured plane. For $\CLE$ in the punctured plane, there is no loop separating the origin from infinity, and we define the gasket to be the set of points that are not separated by any loop from infinity (or equivalently, not separated by any loop from the origin). For $\CLE$ in punctured plane, the invariance under inversion $z\mapsto 1/z$ is true by construction (which is not trivially true for nested $\CLE$ in whole plane \cite{KemppainenWernerNestedSimpleCLERiemannSphere}).

We use the name ``$\CLE$ in doubly connected regions" to indicate the above three versions of $\CLE$: $\CLE$ in the annulus, $\CLE$ in the punctured disc, and $\CLE$ in the punctured plane.

In \cite{SheffieldWernerCLE}, the authors describe an exploration procedure to discover the loops in $\CLE$ progressively. The conformal invariance and domain Markov property of $\CLE$ make this exploration procedure easy to control. In our paper, we use the same procedure to explore the loops in $\CLE$ in the punctured disc.
We will give a precise quantitative relation between the continuous exploration process of $\CLE$ in the punctured disc and the continuous exploration process of $\CLE$ in the unit disc. The authors are in the process of carrying out a program to define a conformal invariant distance on $\CLE_4$ and other $\CLE$ loop configurations which includes \cite{WernerWuCLEExploration, WangWuLevellinesGFFI, WangWuLevellinesGFFII, SheffieldWatsonWuMetric}.  The continuous exploration process is an important ingredient in describing the ``distance" between loops, and this quantitative relation between exploration processes would shed lights on the asymptotic of the ``distance".
\smallbreak
\noindent\textbf{Acknowledgments.} The authors acknowledge Wendelin Werner, Greg Lawler, Julien Dub\'edat and Jason Miller for useful discussion on this project. S.\ Sheffield is funded by NSF DMS-1209044. S.\ Watson's is funded by the NSF Graduate Research Fellowship Programm, award No. 1122374. H.\ Wu's work is funded by NSF DMS-1406411.
\smallbreak
\noindent\textbf{Outline.} In Section \ref{sec::preliminiaries}, we give preliminaries about $\CLE$ in the unit disc and other tools.
We construct and study $\CLE$ in the annulus in Section \ref{sec::cle_annulus}, $\CLE$ in the punctured disc in Section \ref{sec::cle_punctured_disc}, and $\CLE$ in the punctured plane in  Section \ref{sec::cle_punctured_plane}.

\section{Preliminaries}
\label{sec::preliminiaries}
In this paper, we denote the disc, circle and annulus as follows: for $0<r<R$, $x\in\C$,
\begin{align*}
B(x,r)=\{z\in\C: |z-x|<r\},\quad & \D=B(0,1),\\
C(x,r)=\{z\in\C: |z-x|=r\},\quad & C_r=C(0,r),\\
\A(r,R)=\{z\in\C: r<|z|<R\},\quad  & \A_r=\A(r,1).
\end{align*}
\noindent Denote the punctured disc and punctured plane in the following way
\[\D^{\dag}=\D\setminus\{0\},\quad \C^{\dag}=\C\setminus\{0\}.\]
\noindent Throughout the paper, we fix the following constants:
\begin{equation}\label{eqn::universal_relation_constants}
\kappa\in \left(\frac{8}{3},4\right],\quad \beta=\frac{8}{\kappa}-1,\quad \alpha=\frac{(8-\kappa)(3\kappa-8)}{32\kappa},\quad c=\frac{(6-\kappa)(3\kappa-8)}{2\kappa}.
\end{equation}
\noindent For general positive functions $f$ and $g$, we write $f\lesssim g$ if $f/g$ is bounded from above by some universal constant; $f\gtrsim g$ if $g\lesssim f$; and $f\asymp g$ if $f\lesssim g $ and $f\gtrsim g$.

\subsection{Conformal radius and conformal modulus}
In this section, we are interested in two kinds of domains: non-trivial simply connected domains and annuli.

A non-trivial simply connected domain $D$ is a non-empty open subset of $\C$, which is not all of $\C$, such that both $D$ and its complement in the Riemann sphere are connected. From the Riemann mapping theorem, we know that, for any non-trivial simply connected domain $D$ and an interior point $z\in D$, there exists a unique conformal map $\Phi$ from $D$ onto the unit disc $\D$ such that $\Phi(z)=0$ and $\Phi'(z)>0$. We define the \textbf{conformal radius} of $D$ seen from $z$ as
\[\CR(D;z)=1/\Phi'(z).\]
We write $\CR(D)=\CR(D;z)$ if $z=0$.

Consider a closed subset $K$ of $\D$ such that $\D\setminus K$ is simply connected and $0\in\D\setminus K$. There exists a unique conformal map $\Phi_K$ from $\D\setminus K$ onto $\D$ normalized at the origin: $\Phi_K(0)=0$, and $\Phi'_K(0)>0$. In fact $\Phi'_K(0)\ge 1$, and
$$\CR(\D\setminus K)=1/\Phi'_K(0)\le 1.$$
The Schwarz lemma and the Koebe one quarter theorem imply that
\begin{equation}
d\le \CR(\D\setminus K)\le 4d
\end{equation}
where $d=\dist(0,K)$ is the Euclidean distance from the origin to $K$.

Define the \textbf{capacity of $K$ in $\D$ seen from the origin} as
\[\capa(K)=-\log \CR(\D\setminus K)\ge 0.\]
By convention, if $0\in K$, we set $\CR(\D\setminus K)=0$ and $\capa(K)=\infty$. When $K$ is small, i.e. the diameter $R(K)$ of $K$ is less than $1/2$, we have that\footnote{We may assume $K$ is contained in $B(1,R(K))$. Then $\capa(K)\le \capa(B(1,R(K))\cap\D)\asymp R(K)^2$.}
\[\capa(K)\lesssim R(K)^2.\]

An annular domain $A$ is a connected open subset of $\C$ such that its complement in the Riemann sphere has two connected components and both of them contain more than one point. Then there exists a unique constant $r\in (0,1)$ such that $A$ can be conformally mapped onto the standard annulus $\A_r$.
We define the \textbf{conformal modulus} of $A$, denoted as $\CM(A)$, to be this unique constant $r$.
Note that two annuli with different conformal radii can not be conformally mapped onto each other.

The following lemma describes the relation between the conformal radius of a non-trivial simply connected domain and the conformal modulus of an annulus.

\begin{lemma}\label{lem::CR_cvg}
Suppose $K$ is a closed subset of $\D$ such that $\D\setminus K$ is simply connected and $0\in\D\setminus K$. Clearly $\A_r\setminus K$ is an annulus for $r$ small enough. We have that
\[\frac{\CM(\A_r\setminus K)}{\CM(\A_r)}\to \CR(\D\setminus K)^{-1},\quad \text{as }r\to 0.\]
\end{lemma}
\begin{proof}
Suppose $\Phi$ is the conformal map from $\D\setminus K$ onto $\D$ normalized at the origin: $\Phi(0)=0$, and $\Phi'(0)>0$. 
Note that $\Phi(\A_r\setminus K)$ equals $\D\setminus \Phi(r\overline{\D})$, the inner hole of which is asymptotically a disk of radius $\Phi'(0)r$ as $r\to 0$. Thus we have that
\[\lim_{r\to 0}\frac{\CM(\A_r\setminus K)}{\CM(\A_r)}=\lim_{r\to 0}\frac{1}{r}\CM(\Phi(\A_r\setminus K))=\Phi'(0).\]
This implies the conclusion.
\end{proof}

\subsection{Brownian loop soup}\label{subsec::preliminaries_brownianloopsoup}
We now briefly recall some results from \cite{LawlerWernerBrownianLoopsoup}.
It is well known that Brownian motion in $\C$ is conformal invariant. Let us now define, for all $t \ge 0$, the law $\mu_t (z,z)$ of the two-dimensional Brownian bridge of time length
$t$ that starts and ends at $z$ and define
\[\mu^{\mathrm{loop}}=\int_{\C} \int_0^\infty  d^2z \frac {dt}{t} \mu_t (z,z)\]
where $d^2z$ is the Lebesgue measure in $\C$. We stress that $\mu^{\mathrm{loop}}$ is a measure on \textit{unrooted} loops modulo time-reparameterization (see \cite{LawlerWernerBrownianLoopsoup}). And $\mu^{\mathrm{loop}}$ inherits a striking conformal invariance property.
Namely, if for any subset $D\subset\C$, one defines the Brownian loop measure $\mu^{\mathrm{loop}}_D$ in $D$ as the restriction of $\mu^{\mathrm{loop}}$ to the set of loops contained in $D$, then it is shown in \cite{LawlerWernerBrownianLoopsoup} that:
\begin{itemize}
\item For two domains $D'\subset D$, $\mu^{\mathrm{loop}}_D$ restricted to the loops contained in $D'$ is the same as $\mu^{\mathrm{loop}}_{D'}$ (this is a trivial consequence of the definition of these measures).
\item For two connected domains $D_1,D_2$, suppose $\Phi$ is a conformal map from $D_1$ onto $D_2$, then the image of $\mu^{\mathrm{loop}}_{D_1}$ under $\Phi$ has the same law as $\mu^{\mathrm{loop}}_{D_2}$ (this non-trivial fact is inherited from the conformal invariance of planar Brownian motion).
\end{itemize}

Suppose $D$ is a domain and $V_1,V_2$ are two subsets of $D$. We denote by
\[\Lambda(V_1,V_2;D)\]
the measure of the set of Brownian loops in a domain $D$ that intersect both $V_1$ and $V_2$.

\begin{result} \cite[Lemma 3.1, Equation (22)]{LawlerBrownianLoopMeasure}\label{res::Brownianloopmeasure_estimate}
Suppose $0<r<1, R\ge 2$. Then we have that
$$\Lambda(C_1,C_R; \C\setminus \D_r)=2\int_r^1 s^{-1}\rho(R/s)ds$$
where the function $\rho$ satisfies the following estimate: there exists a universal constant $C<\infty$ such that, for $u\ge 2$
$$|\rho(u)-\frac{1}{2\log u}|\le \frac{C}{u\log u}.$$
\end{result}

For a fixed domain $D\subset \C$ and a constant $c>0$, a \textbf{Brownian loop-soup with intensity $c$ in $D$} is a Poisson point process with intensity $c\mu^{\mathrm{loop}}_D$. From the properties of Brownian loop measure, we have the following: Fix a domain $D$ and a constant $c>0$, and suppose $D'$ is a subset of $D$. Suppose $\LL$ is a Brownian loop-soup in $D$, let $\LL_1$ be the collection of loops in $\LL$ that are totally contained in $D'$ and let $\LL_2=\LL\setminus \LL_1$. Then $\LL_1$ has the same law as the Brownian loop soup in $D'$, and $\LL_1$ and $\LL_2$ are independent.

\subsection{CLE in the unit disc}
\subsubsection{Definition and properties}\label{subsec::simplecle_def}
A simple loop in the plane is the image of the unit circle under a continuous injective map. The Jordan Theorem says that a simple loop $L$ separates the plane into two connected components that we call its interior $\inte(L)$ (the bounded one) and its exterior (the unbounded one). We will use the $\sigma$-field $\Sigma$ generated by all the events of the type $\{O\subset\inte(L)\}$ where $O$ spans the set of open sets in the plane. Consider (at most countable) collections $\Gamma=(L_j,j\in J)$ of non-nested disjoint simple loops that are locally finite, i.e., for each $\eps>0$, only finitely many loops $L_j$ have a diameter greater than $\eps$. The space of collections of locally finite, non-nested, disjoint simple loops is equipped with the $\sigma$-field generated by the sets $\{\Gamma: \#\Gamma\cap A=k\}$ where $A\in\Sigma$ and $k\ge 0$. Therefore, to characterize the law of $\Gamma$, we only need to characterize the laws of macroscopic loops in $\Gamma$. In other words, if we characterize the law of all loops in $\Gamma$ with diameter greater than $\eps$ for each $\eps>0$, then the law of $\Gamma$ is determined.

Let us now briefly recall some features of $\CLE$ for $\kappa \in (8/3, 4]$ -- we refer to \cite{SheffieldWernerCLE} for details (and the proofs) of these statements.
A $\CLE$ in $\D$ is a collection $\Gamma$ of non-nested disjoint simple loops $(\gamma_j, j \in  J)$ in $\D$ that possesses a particular conformal restriction property. In fact, this property, which we will now recall, characterizes these $\CLE$s:
\begin{itemize}
\item (Conformal Invariance) For any M\"obius transformation $\Phi$ of $\D$ onto itself, the laws of $\Gamma$ and  $\Phi(\Gamma)$ are the same. This makes it possible to define, for any non-trivial simply connected domain $D$ (that can therefore be viewed as the conformal image of $\D$ via some map $\tilde \Phi$), the law of $\CLE$ in $D$ as the distribution of $\tilde \Phi (\Gamma)$ (because this distribution does not depend on the actual choice of conformal map $\tilde \Phi$ from $\D$ onto $D$).
\item (Domain Markov Property) For any non-trivial simply connected domain $D \subset \D$, define the set $D^* = D^* (D, \Gamma)$ obtained by removing from $D$ all the loops (and their interiors) of $\Gamma$ that do not entirely lie in $D$. Then, conditionally on $D^*$, and for each connected component $U$ of $D^*$, the law of those loops of $\Gamma$ that do stay in $U$ is exactly that of a $\CLE$ in $U$.
\end{itemize}

As we mentioned in Section \ref{sec::introduction}, the loops in a given $\CLE$ are $\SLE_{\kappa}$ type loops for some value of $\kappa \in (8/3, 4]$ (and they look locally like $\SLE_{\kappa}$ curves).  In fact for each such value of $\kappa$, there exists exactly one $\CLE$ distribution that has $\SLE_{\kappa}$ type loops.

As explained in \cite{SheffieldWernerCLE}, a construction of these particular families of loops can be given in terms of outer boundaries of outermost clusters of the Brownian loops in a Brownian loop soup with intensity $c\in (0,1]$ which is a function in $\kappa$ given by:
\begin{equation*}
c=c(\kappa)=\frac{(6-\kappa)(3\kappa-8)}{2\kappa}.
\end{equation*}
Throughout the paper, we will denote the law of $\CLE$ in a non-trivial simply connected domain $D$ by $\mu^{\sharp}_D$.

\subsubsection{Exploration of CLE in the unit disc}\label{subsec::simplecle_exploration}
In \cite{SheffieldWernerCLE}, the authors introduce a discrete exploration process of $\CLE$ loop configuration. The conformal invariance and the domain Markov property make the discrete exploration much easier to control.
Consider a $\CLE$ in the unit disc, draw a small disc $B(x,\eps)$ and let $\gamma^{\eps}$ be the loop that intersects $B(x,\eps)$ with largest radius. Define the quantity
\begin{equation}\label{eqn::ueps_definition}
u(\eps)=\PP[\gamma^{\eps}\text{ contains the origin}].
\end{equation}
In fact, $u(\eps)=\eps^{\beta+o(1)}$ as $\eps$ goes to zero where $\beta=8/\kappa-1$.

\begin{result}\label{res::bubble_simple_cle}\cite[Section 4]{SheffieldWernerCLE}
The law of $\gamma^{\eps}$ normalized by $1/u(\eps)$ converges towards a bubble measure, denoted as $\nu^{bub}_{\D;x}$ which we call \textbf{$\SLE$ bubble measure in $\D$ rooted at $x$}. This $\nu^{bub}_{\D;x}$ is an infinite $\sigma$-finite measure, and we have
\begin{enumerate}
\item [(1)] $\nu^{bub}_{\D;x}[\gamma \text{ contains the origin}]=1$;
\item [(2)] For $r$ small enough, $\nu^{bub}_{\D;x}[R(\gamma)\ge r]\asymp r^{-\beta}$ where $R(\gamma)$ is the smallest radius $r$ such that $\gamma$ is contained in $B(x,r)$.
\end{enumerate}
\end{result}

Because of the conformal invariance and the domain Markov property, we can repeat the ``small semi-disc exploration" until we discover the loop containing the origin: Suppose we have a $\CLE$ loop configuration in the unit disc $\D$. We draw a small semi-disc of radius $\eps$ whose center is uniformly chosen on the unit circle. The loops that intersect this small semi-disc are the loops we discovered. If we do not discover the loop containing the origin, we refer to the connected component of the remaining domain that contains the origin as the \textit{to-be-explored domain}. Let $f_1^{\eps}$ be the conformal map from the to-be-explored domain onto the unit disc normalized at the origin. We also define $\gamma_1^{\eps}$ as the loop we discovered with largest radius. Because of the conformal invariance and the domain Markov property of $\CLE$, the image of the loops in the to-be-explored domain under the conformal map $f_1^{\eps}$ has the same law as simple $\CLE$ in the unit disc. Thus we can repeat the same procedure for the image of the loops under $f_1^{\eps}$. We draw a small semi-disc of radius $\eps$ whose center is uniformly chosen on the unit circle. The loops that intersect the small semi-disc are the loops we discovered at the second step. If we do not discover the loop containing the origin, define the conformal map $f_2^{\eps}$ from the to-be-explored domain onto the unit disc normalized at the origin. The image of the loops in the to-be-explored domain under $f_2^{\eps}$ has the same law as $\CLE$ in the unit disc, etc. At some finite step $N$, we discover the loop containing the origin, we define $\gamma^{\eps}_{N}$ as the loop containing the origin discovered at this step and stop the exploration. We summarize the properties and notations in this discrete exploration below.
\begin{itemize}
\item Before $N$, all steps of discrete exploration are i.i.d.
\item The number of the step $N$, when we discover the loop containing the origin, has the geometric distribution:
\[\PP[N>n]=\PP[\gamma^{\eps} \text{does not contain the origin}]^n=(1-u(\eps))^n.\]
\item Define the conformal map
\[\Phi^{\eps}=f^{\eps}_{N-1}\circ \cdots\circ f_2^{\eps}\circ f_1^{\eps}.\]
\end{itemize}
As $\eps$ goes to zero, the discrete exploration will converge to a Poisson point process of bubbles with intensity measure given by
\[\nu^{bub}_{\D}=\int_{\partial\D}dx \nu^{bub}_{\D;x}\]
where $dx$ is Lebesgue length measure on $\partial\D$. See \cite{SheffieldWernerCLE} for details.

Now we can reconstruct $\CLE$ loops from the Poisson point process of $\SLE$ bubbles. Let $(\gamma_t,t\ge 0)$ be a Poisson point process with intensity $\nu^{bub}_{\D}$.
Namely, let $((\gamma_j,t_j),j\in J)$ be a Poisson point process with intensity $\nu^{bub}_{\D}\times[0,\infty)$, and then arrange the bubble according to the time $t_j$, i.e. denote $\gamma_t$ as the bubble $\gamma_j$ if $t=t_j$, and $\gamma_t$ is empty set if there is no $t_j$ that equals $t$. Clearly, there are only countably many bubbles in $(\gamma_t,t\ge 0)$ that are not empty set.
Define
\[\tau=\inf\{t: \gamma_t\text{ contains the origin}\}.\]
For each $t<\tau$, the bubble $\gamma_t$ does not contain the origin. Define $f_t$ to be the conformal map from the connected component of $\D\setminus\gamma_t$ containing the origin onto the unit disc and normalized at the origin: $f_t(0)=0,f_t'(0)>0$. For this Poisson point process, we have the following properties \cite[Section 4, Section 7]{SheffieldWernerCLE}:

\begin{itemize}
\item The time $\tau$ has the exponential law: $\PP[\tau>t]=e^{-t}$.
\item For $r>0$ small, let $t_1(r),t_2(r), ...,t_j(r)$ be the times $t$ before $\tau$ at which the bubble $\gamma_t$ has radius greater than $r$. Define $\Psi^r=f_{t_j(r)}\circ\cdots\circ f_{t_1(r)}$. Then $\Psi^r$ almost surely converges towards some conformal map $\Psi$ in the Carath\'eodory topology seen from the origin  as $r$ goes to zero. And $\Psi$ can be interpreted as $\Psi=\circ_{t<\tau}f_t$.
\item Generally, for each $t\le\tau$, we can define $\Psi_t=\circ_{s<t}f_s$. Then
\[(L_t:=\Psi^{-1}_t(\gamma_t),0\le t\le\tau)\] is a collection of loops in the unit disc and $L_{\tau}$ is a loop containing the origin.
\end{itemize}

The relation between this Poisson point process of bubbles and the discrete exploration process we described above is given via the following result.

\begin{result}\label{res::caratheodory_convergence}\cite[Section 4]{SheffieldWernerCLE}
$\Phi^{\eps}$ converges in distribution to $\Psi$ in the Carath\'eodory topology seen from the origin.
%The collection of loops $(L_t=\Psi^{-1}_t(\gamma_t),0\le t\le\tau)$ has the same law as loops in simple CLE, i.e. removing loops $(L_t, 0\le t\le \tau)$ from $\D$, adding independent simple CLE loops in each connected component, taking the union of added loops and the loops $(L_t, 0\le t\le \tau)$, the obtained collection of loops has the same law as simple CLE loops in the unit disc.
And $L_{\tau}$ has the same law as the loop of $\CLE$ in $\D$ containing the origin.
\end{result}

Write
\[D_t=\Psi_t^{-1}(\D),\quad t\le \tau.\]
We call the sequence of domains $(D_t,t\le\tau)$ \textbf{the continuous exploration process of $\CLE$ in $\D$ (targeted at the origin)}.

\section{CLE in the annulus}
\label{sec::cle_annulus}
\subsection{Definition and properties of CLE in the annulus}
In this section, we will construct $\CLE$ loop configuration in the annulus $\A_r$ for $r\in (0,1)$. We want to use the same idea of constructing $\CLE$ in $\D$ from the Brownian loop soup. Suppose $\LL(\A_r)$ is a Brownian loop soup in $\A_r$. Note that $\LL(\A_r)$ can have clusters that disconnect the inner boundary $C_r$ from the outer boundary $C_1$ and this is the case we will not address in the current paper. We will consider the loop-soup conditioned on the event $E(\LL(\A_r))$ that there is no cluster of $\LL(\A_r)$ that disconnects the inner boundary from the outer boundary.

On the event $E(\LL(\A_r))$, let $\Gamma(\A_r)$ be the collection of the outer boundaries of outermost clusters of $\LL(\A_r)$. Clearly, $\Gamma(\A_r)$ is a collection of disjoint simple loops in $\A_r$. We define \textbf{$\CLE$ in the annulus} $\A_r$ as the law of $\Gamma(\A_r)$ conditioned on the event $E(\LL(\A_r))$. Since the event $E(\LL(\A_r))$ has positive probability, the above $\CLE$ in the annulus is well-defined.

For any annulus $A$, suppose its conformal modulus is $r$ and $\varphi$ is a conformal map from $\A_r$ onto $A$. Then $\CLE$ in the annulus $A$ can be defined as the image of $\CLE$ in the annulus $\A_r$ under the map $\varphi$. And we denote the law of $\CLE$ in the annulus $A$ as $\mu^{\sharp}(A)$.

We denote $p(\A_r)$ as the probability of the event $E(\LL(\A_r))$. Clearly, $p(\A_r)$ only depends on $r$, so we may also denote it by $p(r)$. The following lemma summarizes the asymptotic behavior of $p(r)$ as $r$ goes to zero. Recall the relation in Equation (\ref{eqn::universal_relation_constants}).

\begin{result} \cite[Lemma 7, Corollary 8]{NacuWernerCLECarpets}\label{res::p_asymptotic}
Suppose $p(r)$ is the probability of the event $E(\LL(\A_r))$. Then $p$ is nondecreasing and there exists a universal constant $C<\infty$ such that, for $0<r,r'<1$,
\begin{equation}\label{eqn::aymptoticestimate_p_1}
\frac{1}{C}p(r)p(r'/C)\le p(rr')\le p(r)p(r').
\end{equation}
Furthermore, we have that
\begin{itemize}
\item There exists a constant $C\ge 1$ such that, for $r$ small enough,
\begin{equation}\label{eqn::aymptoticestimate_p_2}r^{\alpha}\le p(r)\le C r^{\alpha}.\end{equation}
\item For any constant $\lambda\in (0,1)$, we have
\begin{equation}\label{eqn::aymptoticestimate_p_3}\lim_{r\to 0}\frac{p(\lambda r)}{p(r)}=\lambda^{\alpha}.\end{equation}
\end{itemize}
\end{result}
\begin{proof} Equation (\ref{eqn::aymptoticestimate_p_1}) is proved in \cite[Lemma 7]{NacuWernerCLECarpets} and Equation (\ref{eqn::aymptoticestimate_p_2}) is proved in \cite[Corollary 8]{NacuWernerCLECarpets}, and we will give a short proof of Equation (\ref{eqn::aymptoticestimate_p_3}). Suppose $\LL_{\lambda r}$ is a Brownian loop soup in the annulus $\A(\lambda,1/r)$ and let $\LL_r$ be the collection of loops in $\LL_{\lambda r}$ that are contained in $\A(1,1/r)$. Denote by $E(\LL_{\lambda r})$ (resp. $E(\LL_r)$) the event that there is no cluster in $\LL_{\lambda r}$ (resp. $\LL_r$) that disconnects the origin from infinity.
Note that
\[\frac{p(\lambda r)}{p(r)}=\frac{p(\A(\lambda,1/r))}{p(\A(1,1/r)}=\PP[E(\LL_{\lambda r})\cond E(\LL_r)],\]
thus the limit of $p(\lambda r)/p(r)$ exists as $r\to 0$. We denote this limit by $f(\lambda)$. Then clearly, for any $\lambda,\lambda'\in (0,1)$, we have that
\[f(\lambda\lambda')=f(\lambda)f(\lambda').\]
This implies that there exists some constant $\alpha'>0$ such that $f(\lambda)=\lambda^{\alpha'}$ for all $\lambda\in (0,1)$. From Equation (\ref{eqn::aymptoticestimate_p_2}), we know that $\alpha'=\alpha$.
\end{proof}
By the conformal invariance of the Brownian loop soup, we have the following:
\begin{proposition}
The $\CLE$ in the annulus $\A(r,\frac{1}{r})$ is invariant under $z\mapsto 1/z$.
\end{proposition}

The following is the annulus version of the domain Markov property:
\begin{proposition}\label{prop::annuluscle_domainmarkov}
Suppose that $\Gamma$ is a $\CLE$ in the annulus $\A_r$, and that $D$ is an open subset of $\A_r$. Let $D^*$ be the set obtained by removing from $D$ all the loops (and their interiors) in $\Gamma$ that are not totally contained in $D$.
\begin{enumerate}
\item [(1)] If $D$ is simply connected, then each connected component of $D^*$ is also simply connected; and given $D^*$, for each connected component $U$ of $D^*$, the conditional law of the loops in $\Gamma$ that stay in $U$ is the same as $\CLE$ in $U$.
\item [(2)] If $D$ is an annulus, then the connected components of $D^*$ can be simply connected or annular; and given $D^*$, for each connected component $U$ of $D^*$, the conditional law of the loops in $\Gamma$ that stay in $U$ is the same as $\CLE$ in $U$.
\end{enumerate}
\end{proposition}
\begin{proof}
We only prove the case when both $D$ and $U$ are annuli. Other cases can be proved similarly.
Let $U_n\subset U$ be an approximation of $U$ whose boundary is a simple path in the lattice $2^{-n}\Z^2$ (see Figure \ref{fig::annuluscle_domainmarkov}). Suppose $F$ is any bounded function on loop configurations that only depends on macroscopic loops (i.e. the loops with diameter greater than $4\times 2^{-n}$).
Then, for any deterministic set $V_n$ such that the probability of $\{U_n=V_n\}$ is strictly positive, we only need to show that, when $\Gamma$ is a $\CLE$ in the annulus $\A_r$, and $\Gamma|_{V_n}$ is the collection of loops of $\Gamma$ that are contained in $V_n$, we have that
\[\mu^{\sharp}_{\A_r}[F(\Gamma|_{V_n}) \cond U_n=V_n]=\mu^{\sharp}_{V_n}[F].\]

\begin{figure}[ht!]
\begin{center}
\includegraphics[width=0.23\textwidth]{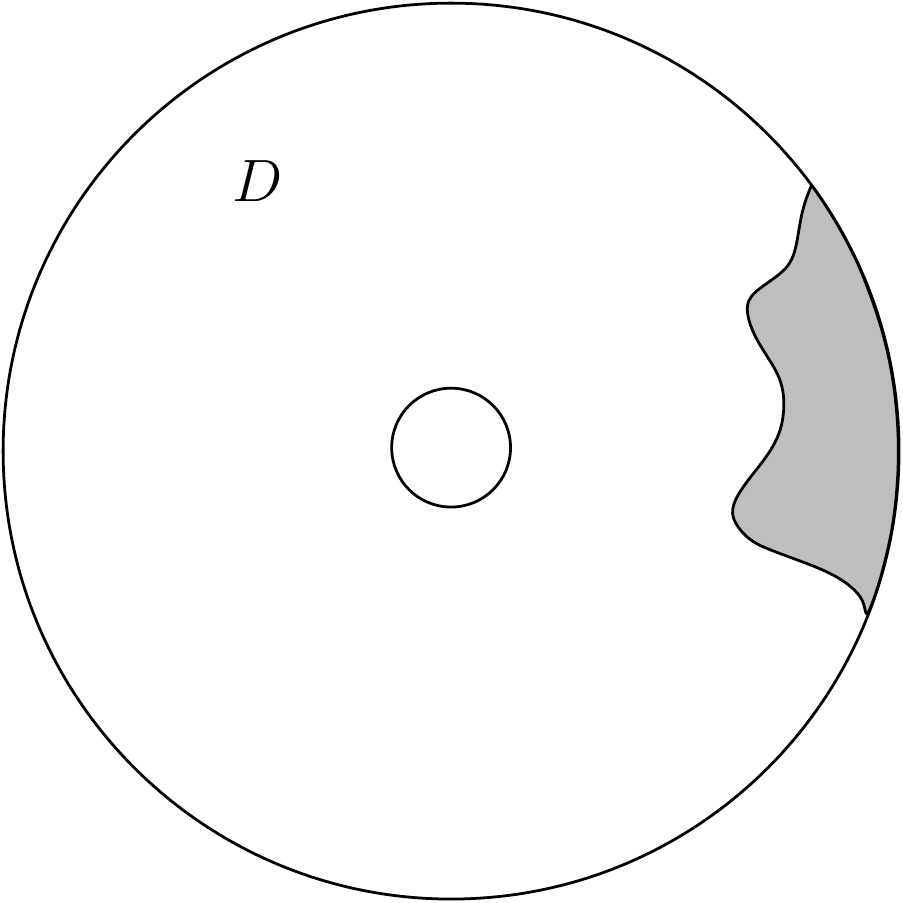}
\includegraphics[width=0.23\textwidth]{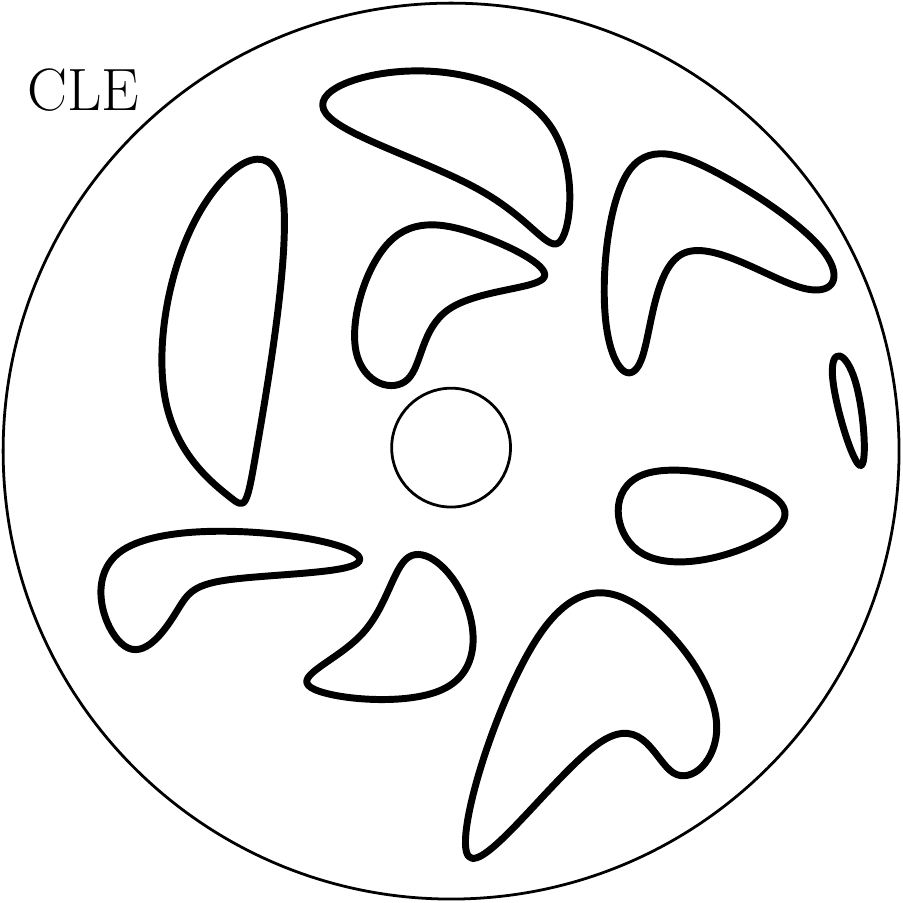}
\includegraphics[width=0.23\textwidth]{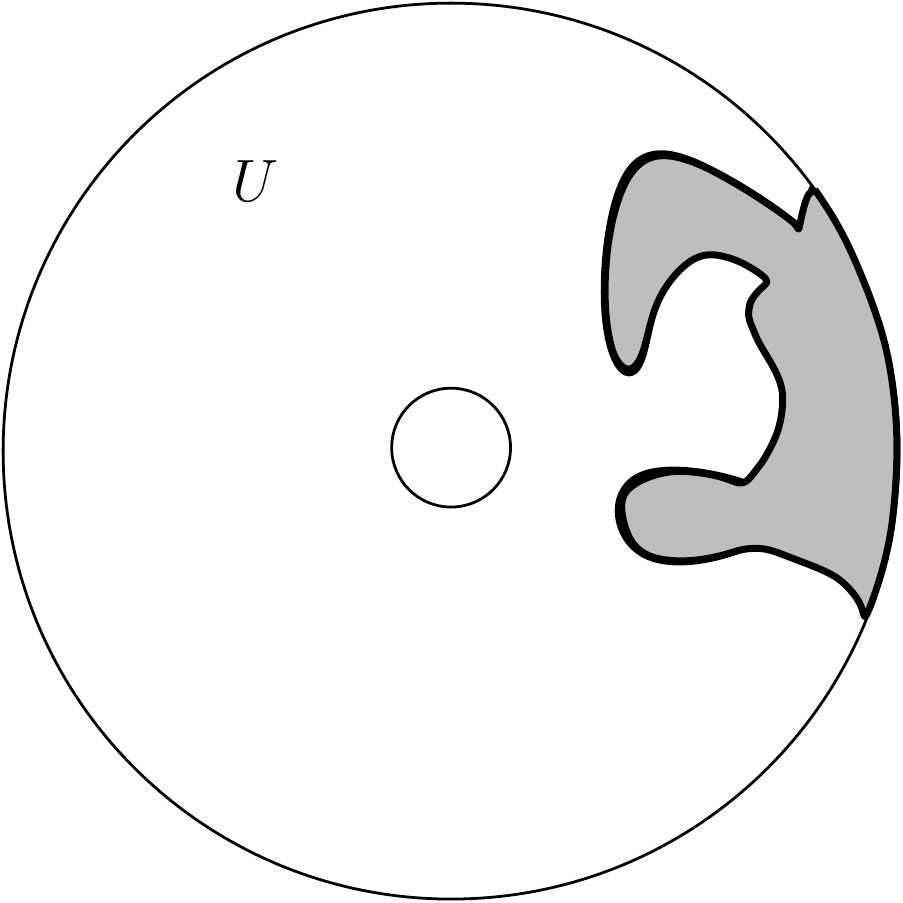}
\includegraphics[width=0.23\textwidth]{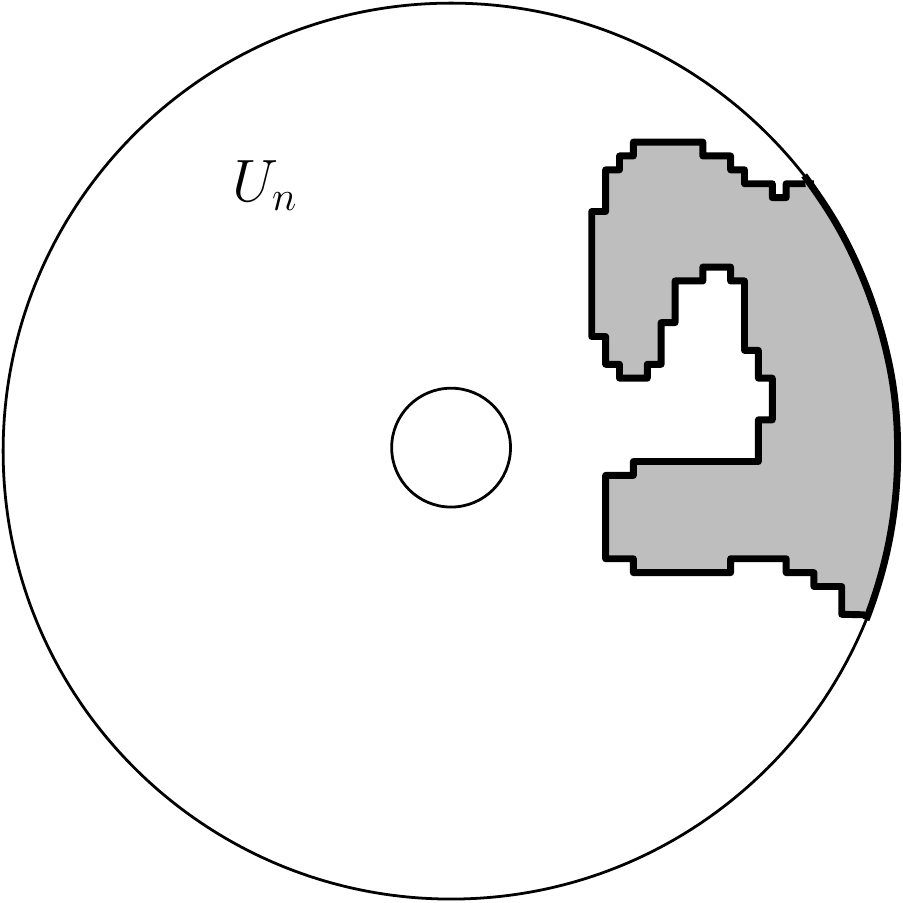}
\end{center}
\caption{\label{fig::annuluscle_domainmarkov} The first panel indicates the annular region $D$. The second panel indicates a $\CLE$ loop configuration in the annulus $\A_r$. The third panel indicates one annular connected component $U$ of $D^*$. The last panel indicates the approximation $U_n$ of $U$.}
\end{figure}

Suppose $\LL$ is a Brownian loop soup in $\A_r$. Let $E(\LL)$ be the event that no cluster of $\LL$ that disconnects $C_r$ from $C_1$, and let $\Gamma(\LL)$ be the collection of outer boundaries of outermost clusters of $\LL$. Then we have that
\begin{eqnarray*}
\lefteqn{\mu^{\sharp}_{\A_r}[F(\Gamma|_{V_n}) 1_{U_n=V_n}]}\\
&=&\E[F(\Gamma(\LL)|_{V_n})1_{U_n=V_n} 1_{E(\LL)}]/p(r)\\
&=&\E[F(\Gamma(\LL)|_{V_n})1_{U_n=V_n} 1_{E_1}1_{E_2}]/p(r),
\end{eqnarray*}
where the events $E_1,E_2$ are defined in the following way:
Consider the loops in $\Gamma(\LL)$ that are contained in $V_n$, the event $E_1$ is that no loop disconnects $C_r$ from $C_1$. Consider the loops in $\Gamma(\LL)$ that are not totally contained in $V_n$, the event $E_2$ is that no loop disconnects $C_r$ from $C_1$. Note that the event $E_1$ is measurable with respect to $\LL|_{V_n}$, which are the loops of $\LL$ that are contained in $V_n$. The event $E_2$ is measurable with respect to the event $\{U_n=V_n\}$ which is independent of $\LL|_{V_n}$. Thus we have
\begin{eqnarray*}
\lefteqn{\E[F(\Gamma(\LL)|_{V_n})1_{U_n=V_n} 1_{E_1}1_{E_2}]/p(r)}\\
&=&\E[\mu^{\sharp}_{V_n}[F]p(V_n)1_{U_n=V_n}1_{E_2}]/p(r)\\
&=&\mu^{\sharp}_{V_n}[F]p(V_n)\PP[U_n=V_n, E_2]/p(r).
\end{eqnarray*}
Thus
\begin{eqnarray*}
\lefteqn{\mu^{\sharp}_{\A_r}[F(\Gamma|_{V_n}) \cond U_n=V_n]}\\
&=&\frac{\mu^{\sharp}_{\A_r}[F(\Gamma|_{V_n})1_{U_n=V_n}]}{\mu^{\sharp}_{\A_r}[1_{U_n=V_n}]}\\
&=&\frac{\mu^{\sharp}_{V_n}[F]p(V_n)\PP[U_n=V_n, E_2]/p(r)}{p(V_n)\PP[U_n=V_n, E_2]/p(r)}\\
&=&\mu^{\sharp}_{V_n}[F].
\end{eqnarray*}
\end{proof}

Propositions \ref{prop::simple_annulus_cle_domainmarkov} and \ref{prop::simple_annulus_cle_origin} describe two ways to find $\CLE$ in annuli from $\CLE$ in simply connected domains.

\begin{proposition}\label{prop::simple_annulus_cle_domainmarkov}
Suppose $\Gamma$ is a $\CLE$ in $\D$ and $D\subset \D$ is an annulus. Let $D^*$ be the set obtained by removing from $D$ all the loops (and their interiors) in $\Gamma$ that are not totally contained in $D$. Note that the connected components of $D^*$ can be simply connected or annular. Then given $D^*$, for each connected component $U$ of $D^*$, the conditional law of the loops in $\Gamma$ that stay in $U$ is the same as $\CLE$ in $U$.
\end{proposition}
\begin{proof}
We only prove the case when $U$ is an annulus. Suppose $\LL$ is a Brownian loop soup in $\D$, and let $\Gamma$ be the collection of the outer boundaries of outermost clusters of $\LL$. Then $\Gamma$ has the law of simple $\CLE$ in $\D$. Suppose $U_n\subset U$ is the approximation of $U$ whose boundary is a simple path in the lattice $2^{-n}\Z^2$ (see Figure \ref{fig::simple_annulus_cle_domainmarkov}). Suppose $F$ is any bounded function on loop configurations that only depends on macroscopic loops (i.e. the loops with diameter greater than $4\times 2^{-n}$). Then, for any deterministic annular set $V_n$ such that the probability of $\{U_n=V_n\}$ is strictly positive, we only need to show that
\begin{equation}\label{eqn::simple_annulus_cle_domainmarkov_goal}
\E[F(\Gamma|_{V_n})\cond U_n=V_n]=\mu^{\sharp}_{V_n}[F],\end{equation}
where $\Gamma|_{V_n}$ is the collection of loops of $\Gamma$ that are contained in $V_n$.

\begin{figure}[ht!]
\begin{center}
\includegraphics[width=0.23\textwidth]{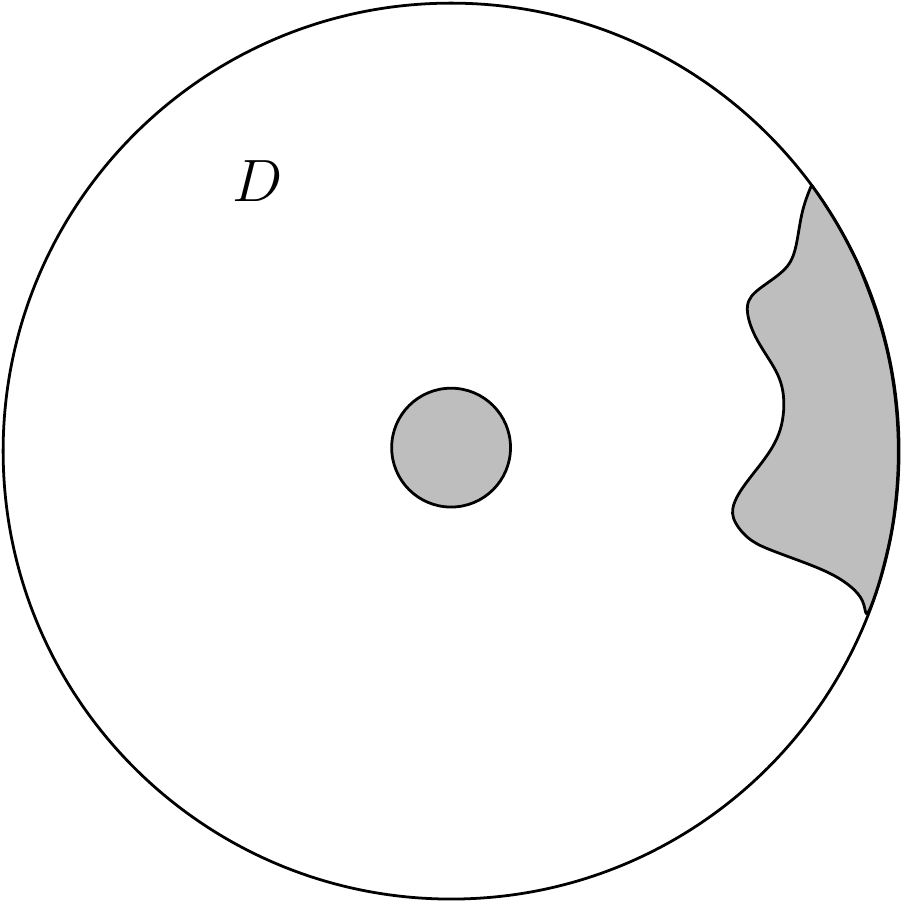}
\includegraphics[width=0.23\textwidth]{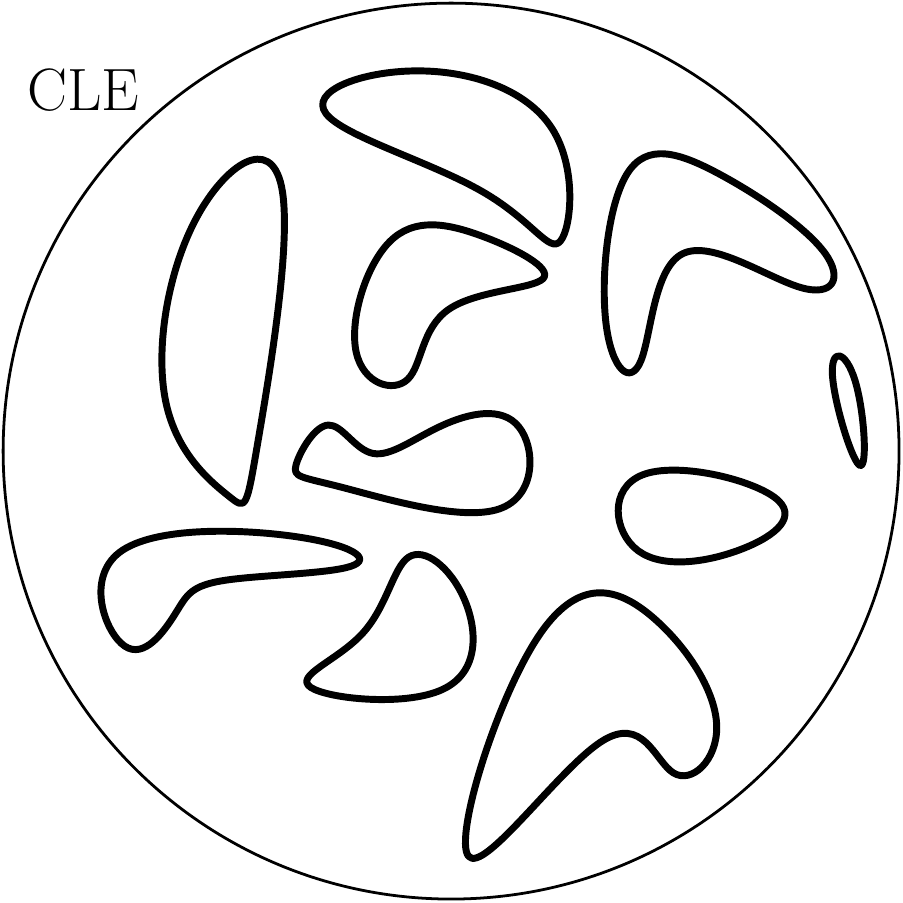}
\includegraphics[width=0.23\textwidth]{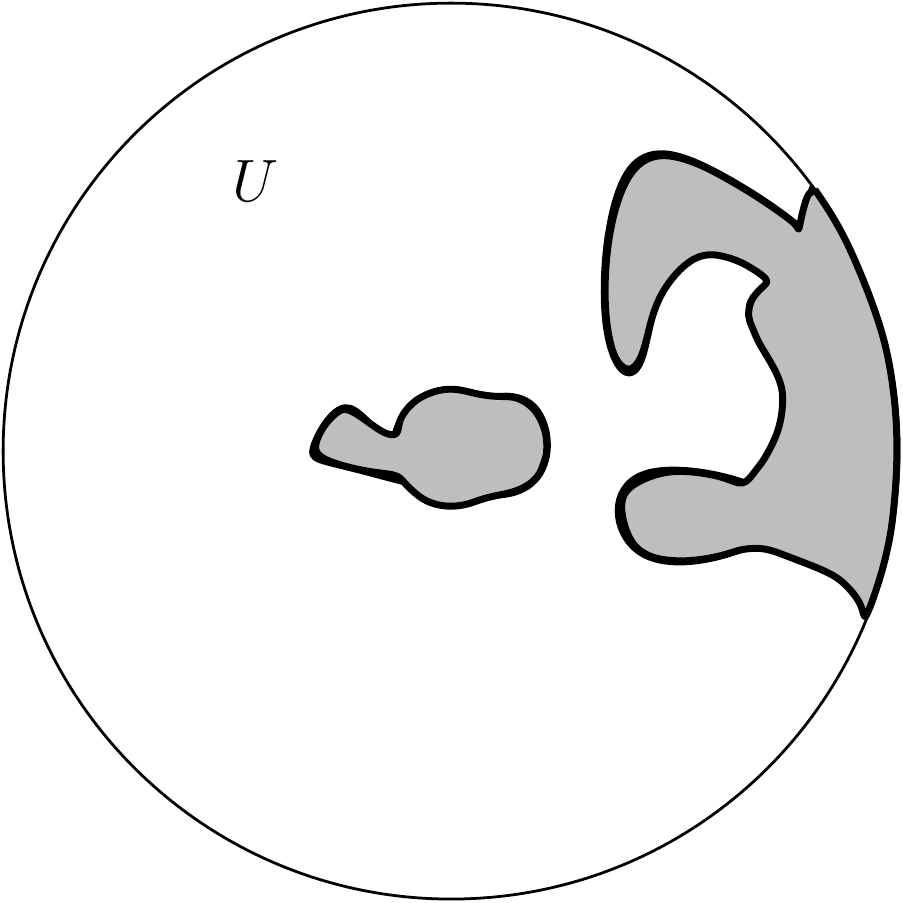}
\includegraphics[width=0.23\textwidth]{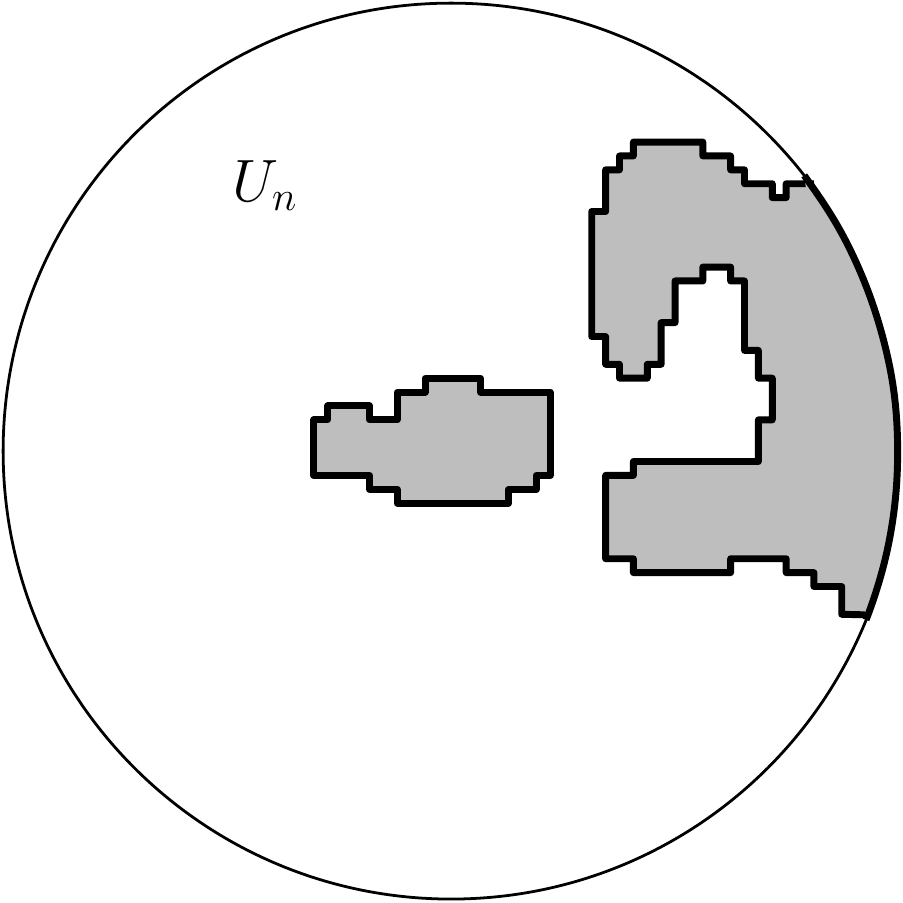}
\end{center}
\caption{\label{fig::simple_annulus_cle_domainmarkov} The first panel indicates the annulus region $D$. The second panel indicates a simple $\CLE$ loop configuration. The third panel indicates one annulus connected component $U$ of $D^*$. The last panel indicates the approximation $U_n$ of $U$.}
\end{figure}

Denote by $\LL|_{V_n}$ the collection of loops in $\LL$ that are totally contained in $V_n$; and denote by $E(\LL|_{V_n})$ the event that there is no loop cluster in $\LL|_{V_n}$ disconnecting the two components of the boundary of $V_n$. Then we can see that, given $E(\LL|_{V_n})$, the event $[U_n=V_n]$ is conditionally independent of $\LL|_{V_n}$. Thus we have that
\[\E[F(\Gamma|_{V_n})1_{U_n=V_n}]=\mu^{\sharp}_{V_n}[F]p(V_n)\PP[U_n=V_n\cond E(\LL|_{V_n})].\]
This implies Equation (\ref{eqn::simple_annulus_cle_domainmarkov_goal}).
\end{proof}

\begin{proposition}\label{prop::simple_annulus_cle_origin}
Suppose $\Gamma$ is a $\CLE$ in $\D$ and $\gamma(0)$ is the loop in $\Gamma$ that contains the origin. Let $D^*$ be the subset of $\D$ obtained by removing from $\D$ the loop $\gamma(0)$ and its interior. Then given $D^*$, the conditional law of the loops in $\Gamma$ that stay in $D^*$ is the same as $\CLE$ in the annulus $D^*$.
\end{proposition}
\begin{proof}
The conclusion can be derived by setting $D=\A_r$ in Proposition \ref{prop::simple_annulus_cle_domainmarkov} and then letting $r$ go to zero.
\end{proof}

\subsection{The SLE bubble measure in the annulus}
Suppose $\Gamma$ is a simple $\CLE$ in $\D$. Recall the discrete exploration of $\Gamma$: we fix $x\in\partial\D$, and explore the loops in $\Gamma$ that intersect $B(x,\eps)$, suppose $\gamma^{\eps}$ is the loop we discovered with largest radius. The probability of the event that $\gamma^{\eps}$ surrounds the origin is $u(\eps)=\eps^{\beta+o(1)}$ and the law of $\gamma^{\eps}$ normalized by $1/u(\eps)$ converges to the bubble measure $\nu^{bub}_{\D;x}$ (recall Proposition \ref{res::bubble_simple_cle}).

We use a similar idea to define the bubble measure of $\CLE$ in the annulus. Fix $r\in (0,1)$ and suppose $\Gamma_r$ is a $\CLE$ in the annulus $\A_r$. We fix $x\in\partial\D$ and explore the loops in $\Gamma_r$ that intersect $B(x,\eps)$; suppose $\gamma_r^{\eps}$ is the loop we discovered with largest radius. Then we have the following conclusion which is a counterpart of Proposition \ref{res::bubble_simple_cle} for $\CLE$ in the annulus (recall the definitions of $u(\eps)$ in Equation (\ref{eqn::ueps_definition}) and the constant $c$ in Equation (\ref{eqn::universal_relation_constants})):
\begin{proposition}\label{prop::bubble_annulus_cle}
The law of $\gamma_r^{\eps}$ normalized by $1/u(\eps)$ converges to a bubble measure in $\A_r$, denoted as $\nu^{bub}_{\A_r;x}$ which we call \textbf{$\SLE$ bubble measure in $\A_r$ rooted at $x$}. Furthermore, the Radon-Nikodym derivative between $\nu^{bub}_{\A_r;x}$ and $\nu^{bub}_{\D;x}$ is given by
\[
\frac{\nu^{bub}_{\A_r;x}}{\nu^{bub}_{\D;x}}[d\gamma]= 1_{\{\gamma\subset \A_r\}\cap E(\gamma)}\frac{p(\A_r\setminus\gamma)}{p(\A_r)}\exp(c\Lambda(r\D,\gamma;\D))\]
where $E(\gamma)$ is the event that $\gamma$ does not surround the origin and $\A_r\setminus\gamma$ indicates the subset of $\A_r$ obtained by removing $\gamma$ and its interior from $\A_r$.
\end{proposition}
\begin{proof}
Suppose $\LL$ is a Brownian loop soup in $\D$, and let $\Gamma$ be the collection of the outer boundaries of outermost clusters of $\LL$. Consider the loops in $\Gamma$ that intersect $B(x,\eps)$, let $\gamma^{\eps}$ be the loop with largest radius. Suppose $\LL_1$ is the collection of loops in $\LL$ that are totally contained in $\A_r$. On the event $E(\LL_1)$, let $\Gamma_1$ be the collection of the outer boundaries of outermost clusters of $\LL_1$. Consider the loops in $\Gamma_1$ that intersect $B(x,\eps)$, let $\gamma_1^{\eps}$ be the loop with largest radius. Denote $\LL_2=\LL\setminus\LL_1$. Note that $\LL_1$ and $\LL_2$ are independent.
Then, for any integrable test function $F$, we have that
\begin{eqnarray*}
\lefteqn{\mu^{\sharp}_{\A_r}[F(\gamma_r^{\eps})\exp(-c\Lambda(r\D,\gamma_r^{\eps};\D))]}\\
&=&\E[F(\gamma_1^{\eps})\exp(-c\Lambda(r\D,\gamma_1^{\eps};\D))1_{E(\LL_1)}]/p(r)\\
&=&\E[F(\gamma_1^{\eps})\exp(-c\Lambda(r\D,\gamma_1^{\eps};\D))1_{E^1_1}1_{E^2_1}]/p(r)
\end{eqnarray*}
where the events $E^1_1,E_1^2$ are defined in the following way: Consider the loops in $\Gamma_1$ that intersect $B(x,\eps)$, the event $E_1^1$ is that no loop disconnects $C_r$ from $C_1$; consider the loops in $\Gamma_1$ that are totally contained in $\A_r\setminus B(x,\eps)$, the event $E_1^2$ is that no loop disconnects $C_r$ from $C_1$. Note that, given the loops in $\Gamma_1$ that intersect $B(x,\eps)$ and the event $E_1^1$, the event $E_1^2$ has probability $p(D_{1,\eps}^*)$ where $D_{1,\eps}^*$ is the set obtained by removing from $\A_r$ all loops (with their interiors) in $\Gamma_1$ that intersect $B(x,\eps)$. We also know that the quantity $\exp(-c\Lambda(r\D,\gamma_1^{\eps};\D))$ is the probability of the event that no loop in $\LL_2$ that intersects $\gamma_1^{\eps}$, which is equivalent to the event that $\{\gamma^{\eps}=\gamma_1^{\eps}\}$. Thus we have
\begin{eqnarray*}
\lefteqn{\E[F(\gamma_1^{\eps})\exp(-c\Lambda(r\D,\gamma_1^{\eps};\D))1_{E^1_1}1_{E^2_1}]/p(r)}\\
&=&\E[F(\gamma_1^{\eps})\exp(-c\Lambda(r\D,\gamma_1^{\eps};\D))1_{E^1_1}p(D_{1,\eps}^*)]/p(r)\\
&=&\E[F(\gamma_1^{\eps})1_{\{\gamma^{\eps}=\gamma^{\eps}_1\}}1_{E^1_1}p(D_{1,\eps}^*)]/p(r).
\end{eqnarray*}
Note that, when $\eps$ is very small, $D_{1,\eps}^*$ is very close to the set $\A_r\setminus \gamma_1^{\eps}$. We have
\begin{eqnarray*}
\lefteqn{\lim_{\eps\to 0}\frac{1}{u(\eps)p(r)}\E[F(\gamma_1^{\eps})1_{\{\gamma^{\eps}=\gamma^{\eps}_1\}}1_{E^1_1}p(D_{1,\eps}^*)]}\\
&=&\lim_{\eps\to 0} \frac{1}{u(\eps)p(r)} \E[F(\gamma^{\eps})1_{\{\gamma^{\eps}=\gamma^{\eps}_1\}}1_{E^1}p(\A_r\setminus \gamma^{\eps})]\\
&=&\lim_{\eps\to 0} \frac{1}{u(\eps)p(r)} \E[F(\gamma^{\eps})1_{\{\gamma^{\eps}\subset \A_r\}}1_{E^1}p(\A_r\setminus \gamma^{\eps})]\\
&=&\lim_{\eps\to 0} \frac{1}{u(\eps)p(r)} \E[F(\gamma^{\eps})1_{\{\gamma^{\eps}\subset \A_r\}}1_{E(\gamma^{\eps})}p(\A_r\setminus \gamma^{\eps})],
\end{eqnarray*}
where the events $E^1$ and $E(\gamma^{\eps})$ are defined in the following way: consider the loops in $\Gamma$ that intersect $B(x,\eps)$, the event $E^1$ is that no loop disconnects $C_r$ from $C_1$; the event $E(\gamma^{\eps})$ is that $\gamma^{\eps}$ does not disconnect $C_r$ from $C_1$.

Combining all these relations,we have
\begin{eqnarray*}
\lefteqn{\lim_{\eps\to 0}\frac{1}{u(\eps)}\mu^{\sharp}_{\A_r}[F(\gamma_r^{\eps})\exp(-c\Lambda(r\D,\gamma_r^{\eps};\D))]}\\
&=& \lim_{\eps\to 0} \frac{1}{u(\eps)p(r)} \E[F(\gamma^{\eps})1_{\{\gamma^{\eps}\subset \A_r\}}1_{E(\gamma^{\eps})}p(\A_r\setminus \gamma^{\eps})]\\
&=&\nu^{bub}_{\D;x}\left[F(\gamma)1_{\{\gamma\subset \A_r\}}1_{E(\gamma)}\frac{p(\A_r\setminus\gamma)}{p(r)}\right].
\end{eqnarray*}
\end{proof}

\section{CLE in the punctured disc}
\label{sec::cle_punctured_disc}
\subsection{Construction of CLE in the punctured disc}
We are going to define CLE in the punctured disc. Roughly speaking, it is the limit of CLE in the annulus $\A_r$ as the inradius $r$ goes to zero. There is another natural way to define CLE in the punctured disc: the limit of CLE in the disc conditioned on the event that the loop containing the origin has diameter at most $\eps$ as $\eps$ goes to zero.  From Proposition \ref{prop::simple_annulus_cle_origin}, we could check that the two limiting procedures would give the same limit. Thus, we also refer to CLE in the punctured disc as CLE in the unit disc conditioned that the origin is in the gasket.

\begin{lemma}\label{lem::annuluscles_coupling}
There exists a universal constant $C<\infty$ such that the following is true. For any $\delta\in (0,1), 0<r'<r<\delta^2$, and any subset $D\subset\A_{\delta}$, suppose $\Gamma_r$ (resp. $\Gamma_{r'}$) is a $\CLE$ in the annulus $\A_r$ (resp. $\A_{r'}$), and $D^*_r$ (resp. $D^*_{r'}$) is the set obtained by removing from $D$ all loops (and their interiors) of $\Gamma_r$ (resp. $\Gamma_{r'}$) that are not totally contained in $D$. Then there exists a coupling between $\Gamma_r$ and $\Gamma_{r'}$ such that the probability of the event $\{D^*_r=D^*_{r'}\}$ is at least
\[1-C\frac{\log(1/\delta)}{\log(1/r)}.\]
Furthermore, on the event $\{D^*_r=D^*_{r'}\}$, the collection of loops of $\Gamma_r$ restricted to $D_r^*$ is the same as the collection of loops of $\Gamma_{r'}$ restricted to $D_{r'}^*$.
\end{lemma}
\begin{proof}
Suppose $\LL$ is a Brownian loop soup in $\A_{r'}$. Denote by $\LL_1$ the collection of loops of $\LL$ that are totally contained in $\A_r$, and write $\LL_2=\LL\setminus\LL_1$. Note that $\LL_1$ and $\LL_2$ are independent.
On the event $E(\LL)$, define $\Gamma$ (resp. $\Gamma_1$) to be the collection of outer boundaries of outermost clusters of $\LL$ (resp. $\LL_1$). Note that, conditioned on $E(\LL)$, the collection $\Gamma$ (resp. $\Gamma_1$) has the same law as $\CLE$ in the annulus $\A_{r'}$ (resp. $\A_r$). Let $D^*$ (resp. $D_1^*$) be the set obtained by removing from $D$ all loops (and their interiors) of $\Gamma$ (resp. $\Gamma_1$) that are not totally contained in $D$. Clearly
\[D^*\subset D_1^*\subset \A_{\delta}.\]
Here is a simple observation: on the event $E(\LL)$, if there is no loop of $\LL_2$ that intersects $D_1^*$, then we have $D^*=D^*_1$. Define $S(\LL_2,\A_{\delta})$ as the event that there exists loop of $\LL_2$ intersecting $\A_{\delta}$. Thus we have
\begin{eqnarray*}
\lefteqn{\PP[D^*\neq D^*_1, E(\LL)]/p(r')}\\
&\le& \PP[S(\LL_2,\A_{\delta}), E(\LL)]/p(r')\\
&\le& \PP[S(\LL_2,\A_{\delta}), E_1,E_2]/p(r'),
\end{eqnarray*}
where the events $E_1$ and $E_2$ are defined in the following way: the event $E_1$ is that no loop of $\Gamma_1$ that disconnects $C_r$ from $C_1$; consider the loops of $\LL$ that are totally contained in the annulus $\A(r',r)$, the event $E_2$ is that there is no cluster that disconnects $C_{r'}$ from $C_r$. Clearly, the events $E_1$, $E_2$, $S(\LL_2,\A_{\delta})$ are independent, and the probability of $E_1$ (resp. $E_2$) is $p(r)$ (resp. $p(r'/r)$). Thus we have
\begin{eqnarray*}
\lefteqn{\PP[S(\LL_2,\A_{\delta}), E_1,E_2]/p(r')}\\
&\le& \PP[S(\LL_2,\A_{\delta})]p(r)p(r'/r)/p(r')\\
&\lesssim&  \PP[S(\LL_2,\A_{\delta})],
\end{eqnarray*}
where the constant in $\lesssim$ can be decided from Proposition \ref{res::p_asymptotic} and is universal. To complete the proof, we only need to show that
\[\PP[S(\LL_2,\A_{\delta})]\lesssim \frac{\log(1/\delta)}{\log(1/r)}.\]
Note that the event $S(\LL_2,\A_{\delta})$ is the same as the event that there exists a loop in $\LL$ intersecting both $C_r$ and $C_{\delta}$. The latter event has the probability \[1-\exp(-c\Lambda(C_r,C_{\delta};\A_{r'})).\]
From Proposition \ref{res::Brownianloopmeasure_estimate}, we have that
\begin{eqnarray*}
\lefteqn{\PP[S(\LL_2,\A_{\delta})]}\\
&=&1-\exp(-c\Lambda(C_r,C_{\delta};\A_{r'}))\\
&\lesssim&\Lambda(C_r,C_{\delta};\A_{r'})\\
&=&\Lambda(C_r,C_{\delta};\C\setminus r'\D)-\Lambda(C_r,C_1; \C\setminus r'\D)\\
&=&2\int_{r'}^r \frac{1}{s}\left(\rho(\frac{\delta}{s})-\rho(\frac{1}{s})\right)ds\\
&\lesssim& \int_{r'}^r \frac{1}{s}\frac{\log(1/\delta)}{(\log\frac{1}{s})^2} ds\\
&\lesssim& \frac{\log(1/\delta)}{\log(1/r)}.
\end{eqnarray*}\end{proof}

\begin{theorem}\label{thm::conditionedcle_construction}
There exists a unique measure on collections of disjoint simple loops in the punctured disc, which we call \textbf{$\CLE$ in the punctured disc} or \textbf{$\CLE$ in $\D$ conditioned on the event that the origin is in the gasket}, to which $\CLE$ in the annulus $\A_r$ converge in the following sense. There exists a universal constant $C<\infty$ such that for any $\delta>0$, any subset $D\subset A_{\delta}$, suppose $\Gamma^{\dag}$ is a $\CLE$ in the punctured disc and $\Gamma_r$ is a $\CLE$ in the annulus $\A_r$, and $D^{\dag,*}$ (resp. $D^*_r$) is the set obtained by removing from $D$ all loops of $\Gamma^{\dag}$ (resp. $\Gamma_r$) that are not totally contained in $D$, then $\Gamma^{\dag}$ and $\Gamma_r$ can be coupled so that the probability of the event $\{D^{\dag,*}=D^*_r\}$ is at least
\[1-C\frac{\log(1/\delta)}{\log(1/r)}.\]
Furthermore, on the event $\{D^{\dag,*}=D^*_r\}$, the collection of loops of $\Gamma^{\dag}$ restricted to $D^{\dag,*}$ is the same as the collection of loops of $\Gamma_r$ restricted to $D_r^*$.
\end{theorem}

\begin{proof}
Define $r_k$ to be the sequence of positive values so that:
\[\log\log\frac{1}{r_k}=k.\]
Note that $r_k\to 0$ as $k\to \infty$.
For $k\ge 1$, suppose $\Gamma_k$ is a $\CLE$ in the annulus $\A_{r_k}$ and $D_k^*$ is the set obtained by removing from $D$ all loops of $\Gamma_k$ that are not totally contained in $D$.
From Lemma \ref{lem::annuluscles_coupling}, $\Gamma_k$ and $\Gamma_{k+1}$ can be coupled so that the probability of $\{D^*_k\neq D^*_{k+1}\}$ is at most
\[Ce^{-k}\log(1/\delta),\]
and on the event $\{D^*_k=D^*_{k+1}\}$, the collection of loops of $\Gamma_k$ restricted to $D^*_k$ is the same as the collection of loops of $\Gamma_{k+1}$ restricted to $D^*_{k+1}$.
Suppose that, for each $k\ge 1$, $\Gamma_k$ and $\Gamma_{k+1}$ are coupled in this way. Then with probability 1, for all but finitely many couplings, we have that $D^*_k=D^*_{k+1}$. Suppose that this is true for all $k\ge l$, and define, for $k\ge l$,
\[D^{\dag,*}=D^*_k,\]
and define $\Gamma^{\dag}$ restricted to $D^{\dag,*}$ to be the collection of loops of $\Gamma_k$ restricted to $D_k^*$.
Then, for any $k_0\ge 1$, the probability of $\{D^{\dag,*}\neq D_{k_0}^*\}$ is at most
\[\sum_{k\ge k_0}Ce^{-k}\log(1/\delta)\lesssim e^{-k_0}\log(1/\delta).\]
For any $r>0$, suppose $r_{k_0}\le r\le r_{k_0-1}$. Then $\CLE$ in the annulus $\A_r$, denoted by $\Gamma_r$, can be coupled with $\Gamma_{k_0}$ so that the probability of $\{D_r^*\neq D_{k_0}^*\}$ is at most
\[C\frac{\log(1/\delta)}{\log(1/r)},\]
where $D_r^*$ is the set obtained by removing from $D$ all loops of $\Gamma_r$ that are not totally contained in $D$. And on the event $\{D_r^*=D_{k_0}^*\}$, the collection of loops of $\Gamma_r$ restricted to $D_r^*$ is the same as the collection of loops of $\Gamma_{k_0}$ restricted to $D_{k_0}^*$. Therefore, the probability of the event $\{D^{\dag,*}\neq D_r^*\}$ is at most
\[C\log(1/\delta)e^{-k_0}+C\frac{\log(1/\delta)}{\log(1/r)}\lesssim \frac{\log(1/\delta)}{\log(1/r)}.\]
This completes the proof.
\end{proof}

\subsection{Properties of CLE in the punctured disc}

Clearly, $\CLE$ in the punctured disc is invariant under rotation. Thus, it is possible to define $\CLE$ in any non-trivial simply connected domain $D$ with a singular point $z\in D$ via conformal image, and we call it \textbf{$\CLE$ in $D$ conditioned on the event that $z$ is in the gasket}.
Propositions \ref{prop::conditionedcle_domainmarkov_1} and \ref{prop::conditionedcle_domainmarkov_2} describe the domain Markov properties of $\CLE$ in the punctured disc.
\begin{proposition}\label{prop::conditionedcle_domainmarkov_1}
Suppose $\Gamma^{\dag}$ is a $\CLE$ in the punctured disc.
For any subset $D\subset \D$ such that 0 is an interior point of $\D\setminus D$ and that $D$ is either simply connected or an annulus, let $D^{\dag,*}$ be the set obtained by removing from $D$ all loops of $\Gamma^{\dag}$ that are not totally contained in $D$. Then, given $D^{\dag,*}$, for each connected component $U$ of $D^{\dag,*}$, the conditional law of the loops in $\Gamma^{\dag}$ that stay in $U$ is the same as $\CLE$ in $U$.
\end{proposition}
\begin{proof}
The conclusion is a direct consequence of the construction of $\CLE$ in the punctured disc in Theorem \ref{thm::conditionedcle_construction} and the domain Markov property of $\CLE$ in the annulus in Proposition \ref{prop::annuluscle_domainmarkov}.
\end{proof}
\begin{proposition}\label{prop::conditionedcle_domainmarkov_2}
Suppose $\Gamma^{\dag}$ is a $\CLE$ in the punctured disc.
For any simply connected domain $D\subset \D$ such that $0\in D$, let $D^{\dag,*}$ be the set obtained by removing from $D$ all loops of $\Gamma^{\dag}$ that are not totally contained in $D$. Suppose $U$ is the connected component of $D^{\dag,*}$ that contains the origin. Then, given $D^{\dag,*}$, the conditional law of loops in $\Gamma^{\dag}$ that stay in $U$ is the same as $\CLE$ in $U$ conditioned on the event that the origin is in the gasket.
\end{proposition}
\begin{figure}[ht!]
\begin{subfigure}[b]{\textwidth}
\begin{center}
\includegraphics[width=0.23\textwidth]{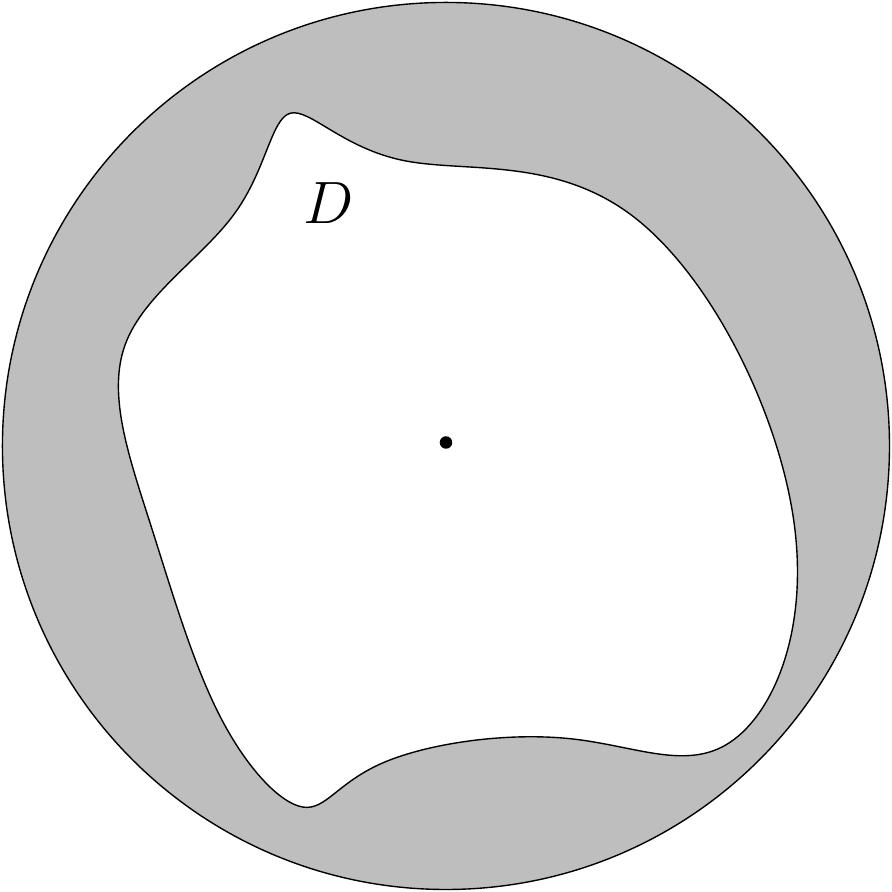}
\includegraphics[width=0.23\textwidth]{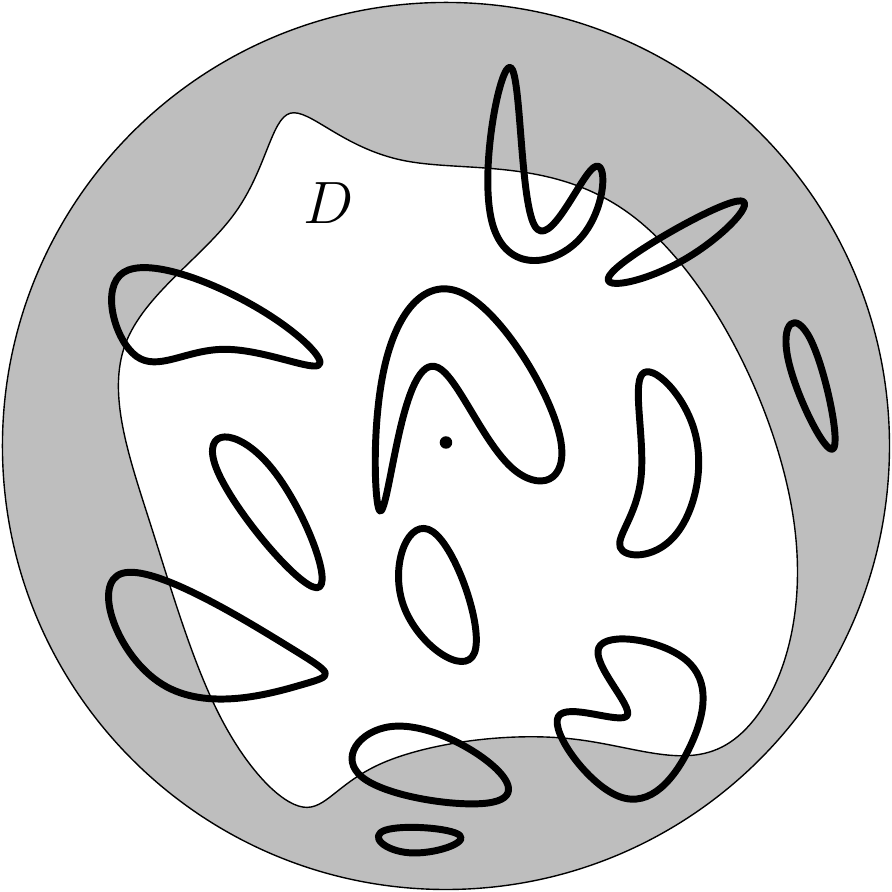}
\includegraphics[width=0.23\textwidth]{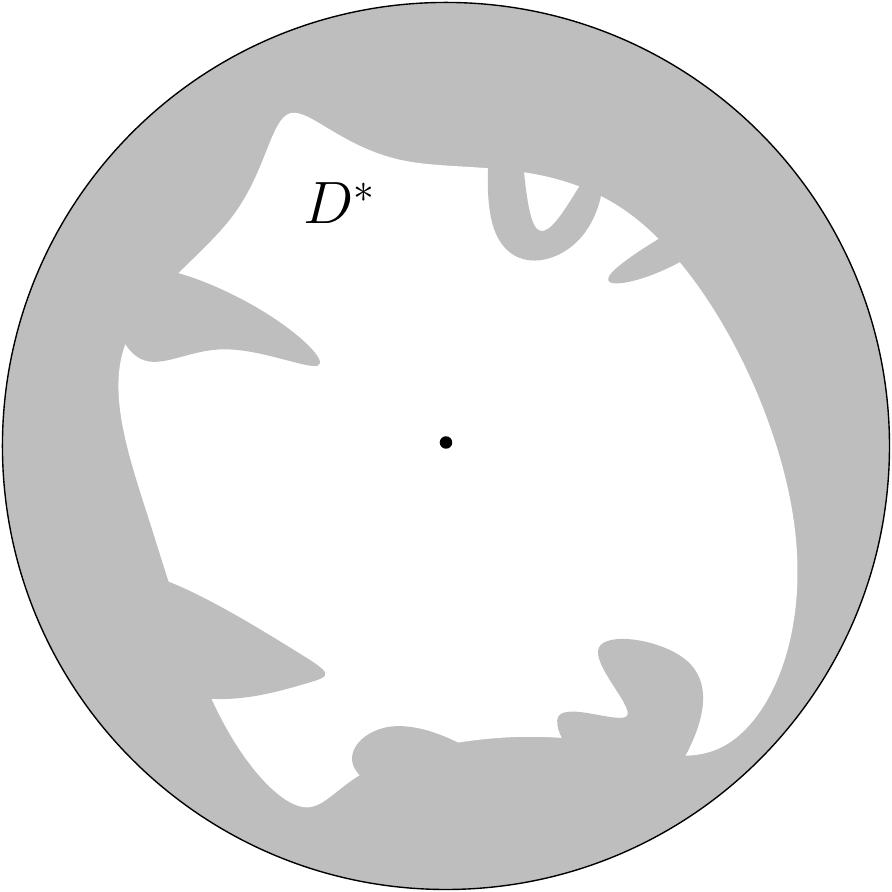}
\includegraphics[width=0.23\textwidth]{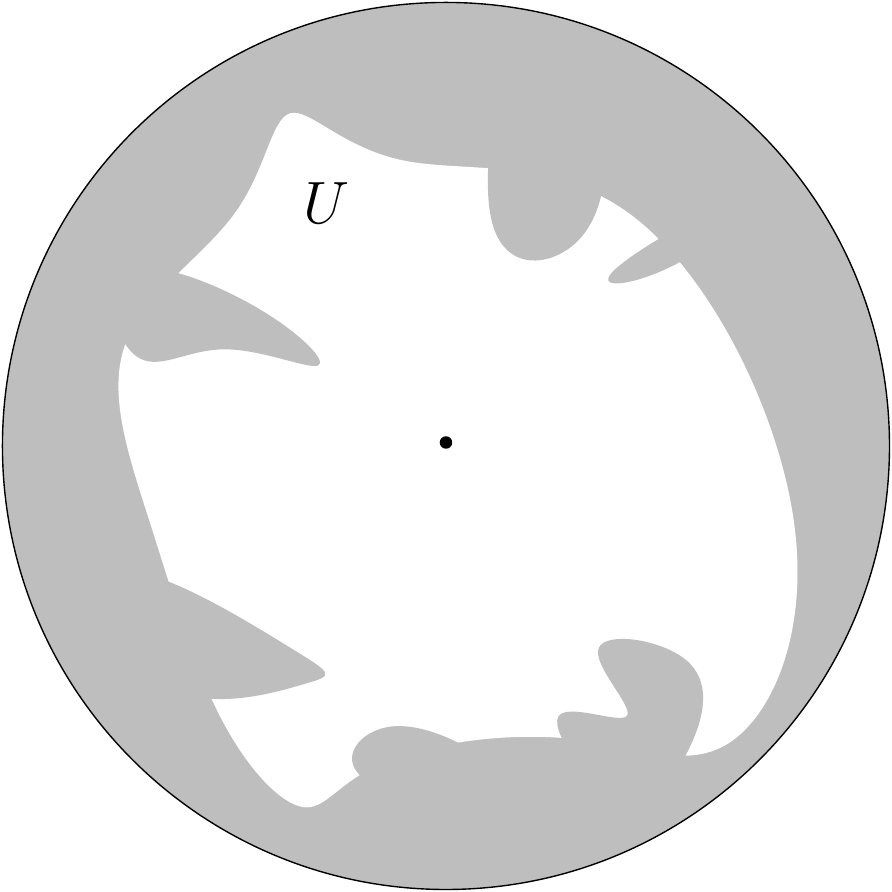}
\end{center}
\caption{The first panel indicates the domain $D$ who contains the origin. The second panel indicates a sample of $\CLE$ in the punctured disc. The third panel indicates the corresponding set $D^*$. The last panel indicates the connected component $U$.}
\end{subfigure}
\begin{subfigure}[b]{\textwidth}
\begin{center}
\includegraphics[width=0.23\textwidth]{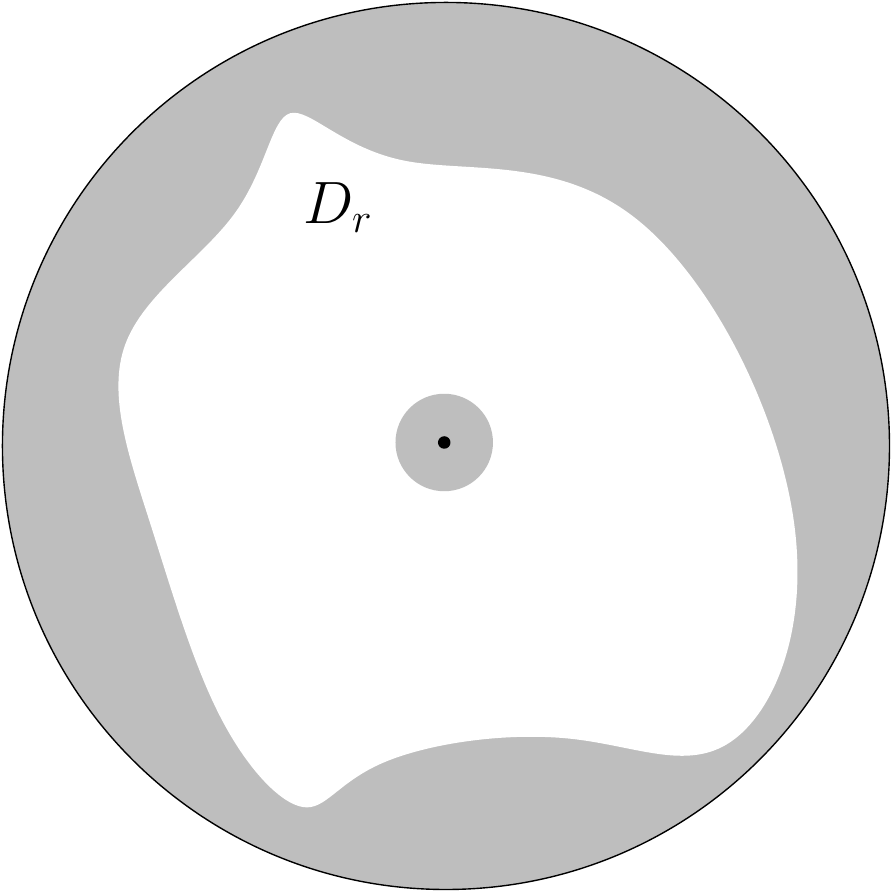}
\includegraphics[width=0.23\textwidth]{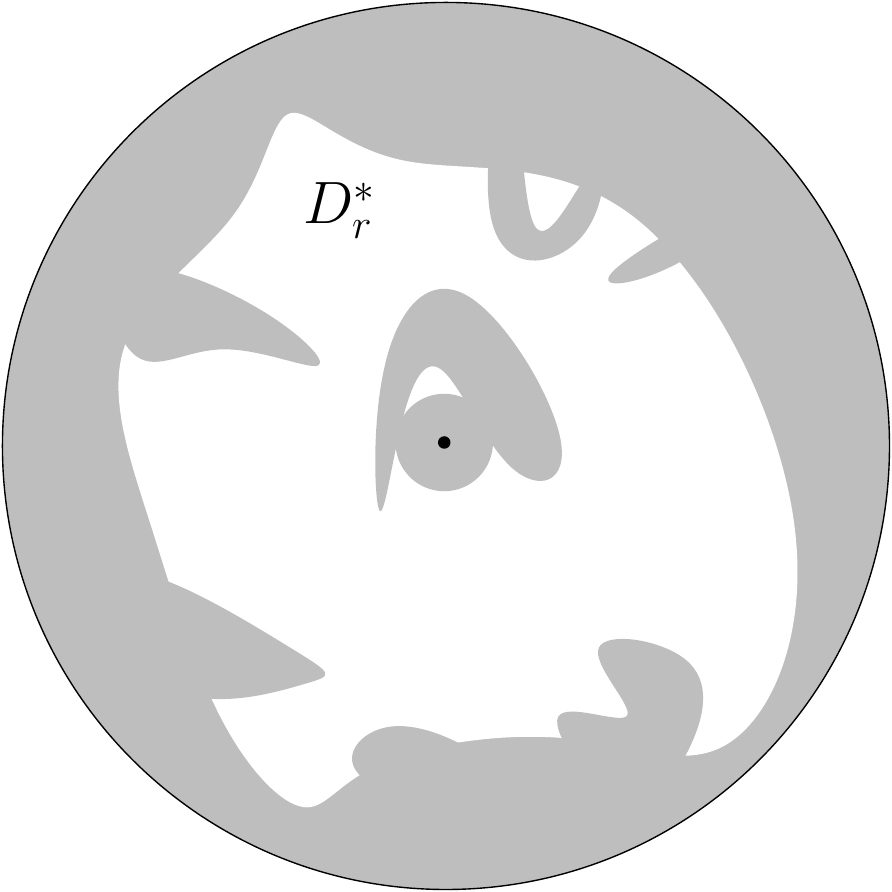}
\includegraphics[width=0.23\textwidth]{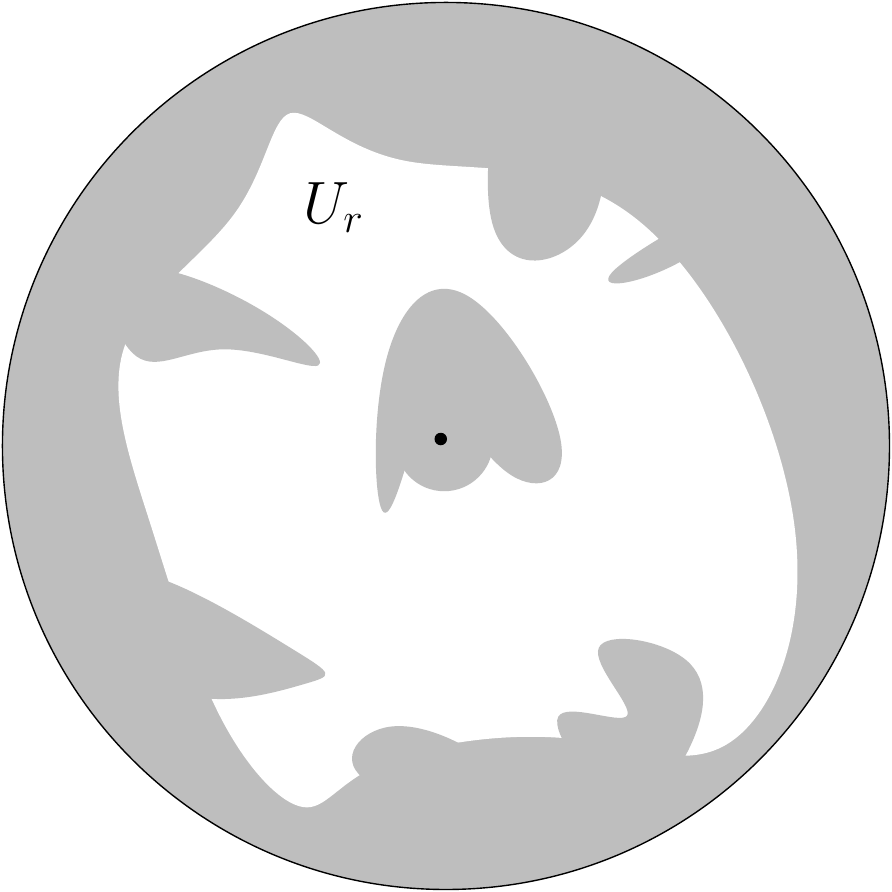}
\end{center}
\caption{The first panel indicates the set $D_r=D\cap \A_r$. The second panel indicates the corresponding $D_r^*$. The last panel indicates the connected component $U_r$.}
\end{subfigure}
\caption{\label{fig::condionedcle_domainmarkov_2}The domain Markov property of $\CLE$ in the punctured disc.}
\end{figure}

\begin{proof}
For $r>0$ small, denote $D_r=D\cap \A_r$, and denote by $D_r^*$ the set obtained by removing from $D_r$ all loops of $\Gamma^{\dag}$ that are not totally contained in $D_r$. Note that, when $r$ is small, it is unlikely that $\Gamma^{\dag}$ has a loop intersecting both $\D\setminus D$ and $r\D$. Suppose there is no such loop and let $U_r$ be the connected component of $D_r^*$ that is contained in $U$ (see Figure \ref{fig::condionedcle_domainmarkov_2}). From Proposition \ref{prop::conditionedcle_domainmarkov_1}, we know that, given $U_r$, the collection of loops of $\Gamma^{\dag}$ restricted to $U_r$ has the same law as $\CLE$ in the annulus. To complete the proof, we only need to point out that, almost surely, $\CM(U_r)\to 0$ as $r\to 0$.
\end{proof}

We will describe the relation between CLE in the disc and CLE in the punctured disc. Roughly speaking, if the loops we are interested are far from the singular point, then these loops in the punctured disc are close to those in the disc. We switch from the unit disc to the upper-half plane so that it is easier to describe ``the loops are far away from the singular point".

\begin{proposition}\label{prop::cle_punctured_disc_simply_connected}
Denote $D=\D\cap\HH$, and let $y>0$ be large. Suppose $\Gamma^{\dag}_y$ is a $\CLE$ in $\HH$ conditioned on the event that $iy$ is in the gasket and $\Gamma$ is a $\CLE$ in $\HH$, and denote $D^{\dag,*}_y$ (resp. $D^*$) as the subset of $D$ obtained by removing from $D$ all loops of $\Gamma^{\dag}_y$ (resp. $\Gamma$) that are not totally contained in $D$.
Then there exists a universal constant $C<\infty$ such that $\Gamma^{\dag}_y$ and $\Gamma$ can be coupled so that the probability of the event $\{D^{\dag,*}_y=D^*\}$ is at least
\[1-\frac{C}{\log y}.\]
Furthermore, on the event $\{D^{\dag,*}_y=D^*\}$, the collection of loops of $\Gamma^{\dag}_y$ restricted to $D_y^{\dag,*}$ is the same as the collection of loops of $\Gamma$ restricted to $D^*$.
\end{proposition}
\begin{proof}
Suppose $\Gamma$ is a simple $\CLE$ in $\HH$ and $\gamma(iy)$ is the loop in $\Gamma$ that contains the point $iy$. In this proof we write $\gamma(iy)$ to refer to the union of the loop and its interior.
We fix a constant $\eta>1/\beta$, and set
\[R=\frac{y}{(\log y)^{\eta}}.\]

From Proposition \ref{prop::simple_annulus_cle_origin}, we know that, given $\gamma(iy)$, the collection of loops in $\Gamma$ restricted to $\HH\setminus \gamma(iy)$, denoted by $\Gamma_1$, has the same law as $\CLE$ in the annulus. Given $\gamma(iy)$ and on the event that $\gamma(iy)\cap C_R=\emptyset$, we have $\{D^*=D^*_1\}$ where $D^*$ (resp. $D^*_1$) is the set obtained by removing from $D$ all loops of $\Gamma$ (resp. $\Gamma_1$) that are not totally contained in $D$.

With the similar idea in the proof of Lemma \ref{lem::annuluscles_coupling}, $\Gamma_1$ can be coupled with $\CLE$ in the annulus $\HH\setminus B(iy,1)$, denoted by $\Gamma_2$, so that the probability of $\{D_1^*\neq D_2^*\}$ is at most
\[C\Lambda(\gamma(iy),C_1;\HH\setminus B(iy,1)),\]
where $D^*_2$ is the set obtained by removing from $D$ all loops of $\Gamma_2$ that are not totally contained in $D$.
On the event $\{\gamma(iy)\cap C_R=\emptyset\}$, this quantity is less than
\[C\Lambda(C_R,C_1;\HH).\]
By \cite[Lemma 4.5]{LawlerBrownianLoopMeasure}, we have that
\[\Lambda(C_R,C_1;\HH)\lesssim \frac{1}{\log R}.\]

From Theorem \ref{thm::conditionedcle_construction}, $\Gamma_2$ can be coupled with $\Gamma^{\dag}_y$ in $\HH\setminus \{iy\}$ so that, the probability of $\{D^*_2\neq D^{\dag,*}_y\}$ is at most
\[\frac{C}{\log y}.\]

Combining all these, we conclude that $\Gamma$ and $\Gamma^{\dag}_y$ can be coupled so that the probability of $\{D^*\neq D^{\dag,*}_y\}$ is less than
\[C\left((\frac{R}{y})^{-\beta+o(1)}+\frac{1}{\log y}+ \frac{1}{\log R}\right)\lesssim \frac{1}{\log y}.\]
This completes the proof.\end{proof}

\subsection{Exploration of CLE in the punctured disc}
We will explore $\CLE$ in the punctured disc in a way similar to the discrete exploration of simple $\CLE$ described in Section \ref{subsec::simplecle_exploration}.
Suppose $\Gamma^{\dag}$ is $\CLE$ in the punctured disc. We explore the loops of $\Gamma^{\dag}$ that intersect $B(x,\eps)$ for some $x\in\partial\D$. Let $\gamma^{\dag,\eps}$ be the loop of $\Gamma^{\dag}$ with the largest radius. Then we have the following conclusion which is a counterpart of Proposition \ref{res::bubble_simple_cle}, recall the definition of $u(\eps)$ in Equation (\ref{eqn::ueps_definition}):
\begin{proposition}\label{prop::bubble_conditioned_cle}
The law of $\gamma^{\dag,\eps}$ normalized by $1/u(\eps)$ converges to a measure, denoted by $\nu^{bub}_{\DO;x}$ which we call the \textbf{$\SLE$ bubble measure in $\DO$ rooted at $x$}. Furthermore, the Radon-Nikodym derivative of $\nu^{bub}_{\DO;x}$ with respect to the $\SLE$ bubble measure $\nu^{bub}_{\D;x}$ in $\D$ rooted at $x$ is given by
\[\frac{\nu^{bub}_{\DO;x}}{\nu^{bub}_{\D;x}}[d\gamma]=1_{E(\gamma)}\CR(\D\setminus\gamma)^{-\alpha}\]
where $E(\gamma)$ is the event that $\gamma$ does not surround the origin and $\D\setminus\gamma$ indicates the subset of $\D$ obtained by removing from $\D$ the bubble $\gamma$ and its interior.
\end{proposition}
\begin{proof}
Combination of Proposition \ref{prop::bubble_annulus_cle} and Lemma \ref{lem::CR_cvg} implies the conclusion.
\end{proof}

Suppose $\gamma$ is a bubble in $\D$ rooted at $x$, and recall that $R(\gamma)$ is the smallest $r$ for which $\gamma$ is contained in $B(x,r)$. Recall the constants defined in Equation (\ref{eqn::universal_relation_constants}); we have the following quantitative results for $\nu^{bub}_{\D;x}$ and $\nu^{bub}_{\DO;x}$:
\begin{lemma}\label{lem::bubble_conditioned_finite}
\[\int 1_{E(\gamma)\cap[R(\gamma)\ge 1/2]}\CR(\D\setminus\gamma)^{-\eta}\nu^{bub}_{\D;x}[d\gamma]<\infty \quad \text{as long as}\quad \eta< 1-\kappa/8.\]
In particular, this implies that
\[\nu^{bub}_{\DO;x}[R(\gamma)\ge 1/2]<\infty.\]
\end{lemma}
\begin{proof}
Conditioned on $\{R(\gamma)>1/2\}\cap E(\gamma)$, we can parameterize the bubble $\gamma$ clockwise by the capacity seen from the origin starting from the root and ending at the root: $(\gamma(t),0\le t\le T)$. Suppose $S$ is the first time that $\gamma$ exits the ball $B(x,1/2)$. Then we know that, given $\gamma[0,S]$, the future part of the curve $\gamma[S,T]$ has the same law as a chordal $\SLE$ in $\D\setminus\gamma[0,S]$ from $\gamma(S)$ to $x$. Thus we only need to show that the integral is finite when we replace the curve by a chordal $\SLE$ curve.

Precisely, suppose $\gamma=(\gamma_t,t\ge 0)$ is a chordal $\SLE$ in the upper-half plane $\HH$ from $0$ to $\infty$ (parameterized by the half-plane capacity). We only need to show that
\begin{equation}\label{eqn::martingale_compensation_estimate}
\E[\CR(\HH\setminus\gamma;i)^{-\eta}]<\infty,
\end{equation}
where $\CR(\HH\setminus\gamma; i)$ is the conformal radius of $\HH\setminus\gamma$ in $\HH$ seen from $i$. This is true for chordal SLE, see \cite[Equation (6.9)]{ViklundLawlerMultifractalSpectrum}). \end{proof}
\begin{lemma}\label{lem::bubble_conditionedcle_quantitative}
\[\int R(\gamma)^{\eta}\nu^{bub}_{\DO;x}[d\gamma]<\infty \quad \text{as long as}\quad \eta>\beta.\]
Note that $\beta\in [1,2)$, thus we have
\[\int R(\gamma)^2 \nu^{bub}_{\DO;x}[d\gamma]<\infty.\]
\end{lemma}
\begin{proof}
We divide the integral into two parts:
\[\int 1_{[R(\gamma)\le 1/2]}R(\gamma)^{\eta}\nu^{bub}_{\DO;x}[d\gamma],\quad \text{and}\quad \int 1_{[R(\gamma)\ge 1/2]}R(\gamma)^{\eta}\nu^{bub}_{\DO;x}[d\gamma].\]
The first part is finite due to Proposition \ref{prop::bubble_conditioned_cle} and Proposition \ref{res::bubble_simple_cle}. The second part is finite by Lemma \ref{lem::bubble_conditioned_finite}.
\end{proof}

Now we can describe the exploration process of $\CLE$ in the punctured disc. Most of the proofs are similar to the proofs in \cite{SheffieldWernerCLE} for simple $\CLE$s. To be self-contained, we rewrite the proofs in the current setting.

%Suppose $\Gamma^0$ is a conditioned CLE in $\DO$. We draw a small semi-disc of radius $\eps$ whose center is uniformly chosen on the unit circle. The loops that intersect this small semi-disc are the loops we discovered. Define the conformal map $f_1^{0,\eps}$ from the remaining domain onto the unit disc and normalizing at the origin. We also define $\gamma_1^{0,\eps}$ as the loop we discovered with largest capacity seen from the origin. Because of the conformal invariance and domain Markov property of conditioned CLE--Proposition \ref{prop::conditionedcle_domainmarkovproperty_2}, the image of the loops in the remaining domain under the conformal map $f_1^{0,\eps}$ has the same law as conditioned CLE in $\DO$. Thus we can repeat the same procedure to the image of the loops under $f_1^{0,\eps}$. We draw a small semi-disc of radius $\eps$ whose center is uniformly chosen on the unit circle. The loops that intersect the small semi-disc are the loops we discovered at the second step. Define the conformal map $f_2^{0,\eps}$ from the remaining domain onto the unit disc and it is normalized at the origin. The image of the loops in the remaining domain under $f_2^{0,\eps}$ has the same law as conditioned CLE in $\DO$, etc. Clearly, all steps of this discrete exploration process are i.i.d. For any $n$, define
%$$\Phi^{0,\eps}_n=f^{0,\eps}_n\circ\cdots\circ f_1^{0,\eps}.$$
%
%As we would expect, as $\eps$ goes to zero, the discrete exploration process will converge to a Poisson point process of bubbles with intensity $M^0$. And we can reconstruct conditioned CLE from these bubbles.
%
Suppose $(\gamma^{\dag}_t,t\ge 0)$ is a Poisson point process with intensity
\[\nu^{bub}_{\DO}=\int_{\partial\D} dx \nu^{bub}_{\DO;x}.\]
For any time $t$, let $f^{\dag}_t$ be the conformal map from $\D\setminus\gamma^{\dag}_t$ onto $\D$ normalized at the origin.
For any fixed $T>0$ and $r>0$, let $t_1(r)<...<t_j(r)$ be the times $t$ before $T$ at which $R(\gamma^{\dag}_t)$ is greater than $r$. Define
\[\Psi^{\dag,r}_T=f^{\dag}_{t_j(r)}\circ\cdots\circ f^{\dag}_{t_1(r)}.\]
Then we have the following:
\begin{lemma}
$\Psi^{\dag,r}_T$ converges almost surely in the Carath\'eodory topology seen from the origin towards some conformal map, denoted as $\Psi^{\dag}_T$, as $r$ goes to zero.
\end{lemma}
\begin{proof}
Lemma \ref{lem::bubble_conditionedcle_quantitative} guarantees that
\begin{eqnarray*}
\lefteqn{\E\left[\sum_{t<T}\capa (\gamma^{\dag}_t)1_{R(\gamma^{\dag}_t)\le 1/2}\right]}\\
&=& T\nu^{bub}_{\DO}[\capa(\gamma^{\dag})1_{\{R(\gamma^{\dag})\le 1/2\}}]\\
&\lesssim& T \nu^{bub}_{\DO}[R(\gamma^{\dag})^2 1_{\{R(\gamma^{\dag})\le 1/2\}}]<\infty.
\end{eqnarray*}
Since there are only finitely many times $t$ before $T$ at which $R(\gamma^{\dag}_t)\ge1/2$, we have that, almost surely,
$$\sum_{t<T}\capa(\gamma^{\dag}_t)<\infty,$$
and this implies the convergence in the Carath\'eodory topology (see \cite[Stability of Loewner chains]{SheffieldWernerCLE}).
\end{proof}
Define $(D_t^{\dag}=(\Psi^{\dag}_t)^{-1}(\D), t\ge 0)$. This is a decreasing sequence of simply connected domains containing the origin, and we call it \textbf{the continuous exploration process of $\CLE$ in the punctured disc}. Write, for $t>0$, $L_t^{\dag}=(\Psi^{\dag}_t)^{-1}(\gamma^{\dag}_t)$. It is clear that
\[D^{\dag}_t=\bigcap_{s\le t}D^{\dag}_s,\quad D^{\dag}_{t+}:=\bigcup_{s>t}D^{\dag}_s=D^{\dag}_t\setminus L^{\dag}_t.\]

%Note that the process$$(\sum_{s<t}\capa(\gamma^{\dag}_s),t\ge 0),$$ viewed as a process in $t$, is a subordinator (without killing). For $\delta>0$ small, define the stopping time
%$$\tau^{\dag}_{\delta}=\inf\{t: \sum_{s<t} \capa(\gamma^{\dag}_s)\ge \log\frac{1}{\delta}\}.$$
%
%Suppose $(\gamma_t,t\ge 0)$ is a Poisson point process with intensity $\nu^{bub}(\D)$ and define $\Psi_t$, $D_t$ in the same way as in Subsection \ref{subsec::simplecle_exploration}. Note that the process
%$$(\sum_{s<t}\capa(\gamma_s),t\ge 0),$$
%is a subordinator with killing. Define
%$$\tau_{\delta}=\inf\{t: \sum_{s<t} \capa(\gamma_s)\ge \log\frac{1}{\delta}\}.$$
%For $t\ge 0$, define
%\begin{equation}\label{eqn::simplecle_exploration_mart}
%M_t=\exp(-\alpha\sum_{s<t}\capa(\gamma_s)+t\int(1-e^{-\alpha \capa(\gamma)})d\nu^{bub}(\D)(\gamma)).
%\end{equation}
%Then it is clear that $(M_t,t\ge 0)$ is a local martingale and we have that
%\begin{proposition}\label{prop::conditionedcle_simplecle_exploration_rn}
%The law of $(D^{\dag}_t,t\le \tau^{\dag}_{\delta})$ is the same as the law of $(D_t,t\le \tau_{\delta})$ weighted by $M_{\tau_{\delta}}$ which is defined through Equation (\ref{eqn::simplecle_exploration_mart}).
%\end{proposition}

Suppose $(\gamma_t,t\ge 0)$ is a Poisson point process with intensity $\nu^{bub}_{\D}$, and $(D_t,t\le \tau)$ is the continuous exploration process of simple $\CLE$ in $\D$ defined from the process $(\gamma_t,t\ge 0)$ as described in Section \ref{subsec::simplecle_exploration}. Define, for $\eta>0$,
\begin{equation}\label{eqn::martingale_compensation}
\theta(\eta):=\int (e^{\eta \capa(\gamma)}-1)1_{E(\gamma)}\nu^{bub}_{\D}[d\gamma],
\end{equation}
where $E(\gamma)$ is the event that $\gamma$ does not surround the origin, and we first assume that $\theta(\alpha)$ is a positive finite constant. Then the relation between the process $(D^{\dag}_t,t\ge 0)$ and the process $(D_t,t\le \tau)$ can be described using the following proposition.
\begin{proposition}\label{prop::continuousexploration_RN}
For any $t>0$, the law of $(\gamma^{\dag}_s,s<t)$ is the same as the law of $(\gamma_s,s<t)$ conditioned on $\{\tau\ge t\}$ and weighted by $M_t$ where
\begin{equation}
M_t=\exp\left(\alpha\sum_{s<t} \capa(\gamma_s)-\theta(\alpha) t\right).
\end{equation}
In particular, for any $t>0$, the law of $D^{\dag}_t$ is the same as the law of $D_t$ conditioned on $\{\tau\ge t\}$ and weighted by $$\CR(D_t)^{-\alpha}e^{-\theta(\alpha) t}.$$
\end{proposition}
\begin{proof}
We first note that the process $(\gamma_s,s<t)$ conditioned on $\{\tau\ge t\}$ has the same law as a Poisson point process with intensity $1_{E(\gamma)}\nu^{bub}_{\D}$ restricted to the time interval $[0,t)$. Suppose $(\hat{\gamma}_s,s\ge 0)$ is a Poisson point process with intensity $1_{E(\gamma)}\nu^{bub}_{\D}$, and define
\[\hat{M}_t=\exp\left(\alpha\sum_{s<t} \capa(\hat{\gamma}_s)-\theta(\alpha) t\right).\]
We only need to show that, for any function $f$ on bubbles that makes every integral finite, we have
\[\E\left[\exp\left(-\sum_{s<t}f(\hat{\gamma}_s)\right)\hat{M}_t\right]=\exp\left(-t\int (1-e^{-f(\gamma)})e^{\alpha \capa(\gamma)}1_{E(\gamma)}\nu^{bub}_{\D}[d\gamma]\right).\]
This can be obtained by direct calculation:
\begin{eqnarray*}
\lefteqn{-\log \E\left[\exp\left(-\sum_{s<t}f(\hat{\gamma}_s)\right)\hat{M}_t\right]}\\
&=& -\log \E\left[\exp\left(-\sum_{s<t}(f(\hat{\gamma}_s)-\alpha \capa(\hat{\gamma}_s))\right)\right]+\theta(\alpha) t\\
&=& t\int (1-e^{-f(\gamma)+\alpha \capa(\gamma)})1_{E(\gamma)}\nu^{bub}_{\D}[d\gamma]+\theta(\alpha) t\\
&=& t\int (1-e^{-f(\gamma)})e^{\alpha \capa(\gamma)} 1_{E(\gamma)}\nu^{bub}_{\D}[d\gamma].
\end{eqnarray*}
\end{proof}

The fact that $\theta(\alpha)$ is a positive finite constant is guaranteed by the following lemma.
\begin{lemma}\label{lem::bubble_estimate}
The quantity
$\theta(\eta)$, which is defined in Equation (\ref{eqn::martingale_compensation}), is finite as long as $\eta< 1-\kappa/8$.

The quantity
\begin{equation}\label{eqn::conditioned_bubble_estimate}
\int(e^{\eta \capa(\gamma)}-1)\nu^{bub}_{\DO}[d\gamma]
\end{equation}
is finite as long as $\eta< 2/\kappa-\kappa/32$.
\end{lemma}
\begin{proof}
The integral in $\theta(\eta)$ may explode when $R(\gamma)$ is small or when $\gamma$ is close to the origin. We will control the two parts separately.

For the first part, we have
\begin{eqnarray*}
\lefteqn{\int (e^{\eta \capa(\gamma)}-1)1_{\{R(\gamma)\le 1/2\}}\nu^{bub}_{\D}[d\gamma]}\\
&\lesssim& \int \capa(\gamma)1_{\{R(\gamma)\le 1/2\}}\nu^{bub}_{\D}[d\gamma]\\
&\lesssim& \int R(\gamma)^2 1_{\{R(\gamma)\le 1/2\}}\nu^{bub}_{\D}[d\gamma]<\infty.
\end{eqnarray*}
For the second part, by Lemma \ref{lem::bubble_conditioned_finite}, we get
\begin{eqnarray*}
\lefteqn{\int (e^{\eta \capa(\gamma)}-1)1_{\{R(\gamma)\ge1/2\}}1_{E(\gamma)}\nu^{bub}_{\D}[d\gamma]}\\
&\le& \int \CR(\D\setminus\gamma)^{-\eta}1_{\{R(\gamma)\ge1/2\}}1_{E(\gamma)}\nu^{bub}_{\D}[d\gamma]<\infty.\end{eqnarray*}

The quantity in Equation (\ref{eqn::conditioned_bubble_estimate}) is finite as long as $\eta+\alpha< 1-\kappa/8$.
\end{proof}

We conclude this section by explaining that the sequence of loops $(L^{\dag}_t,t\ge 0)$ obtained from the sequence of bubbles $(\gamma^{\dag}_t,t\ge 0)$ (the Poisson point process with intensity $\nu^{bub}(\DO)$) does correspond to the loops in $\CLE$ in the punctured disc. Namely, we first remove from $\D$ all loops $L^{\dag}_t$ (with their interiors) for $t\ge 0$, and then, in each connected component, sample independent simple $\CLE$s. We will argue that the collection of these loops from simple $\CLE$ together with the sequence $(L^{\dag}_t,t\ge 0)$ has the same law as the collection of loops in $\CLE$ in the punctured disc. The idea is very similar to the one used in \cite[Section 7]{SheffieldWernerCLE} to show that the loops obtained from the bubbles have the same law as the loops in $\CLE$.

Suppose $\Gamma^{\dag}$ is a $\CLE$ in the punctured disc. Fix a point $z\in\DO$. Let $L^{\dag}(z)$ be the loop in $\Gamma^{\dag}$ that contains $z$. Following is the discrete exploration of $\Gamma^{\dag}$ to discover $L^{\dag}(z)$. Fix $\eps>0$ small and $\delta>\eps$ small. Sample $x_1\in\partial\D$ uniformly chosen from the circle. The loops of $\Gamma^{\dag}$ that intersect $B(x_1,\eps)$ are the loops we discovered. Call the connected component of the remaining domain that contains the origin the to-be-explored domain and let $f^{\dag,\eps}_1$ be the conformal map from the to-be-explored domain onto the unit disc normalized at the origin. Let $\gamma^{\dag,\eps}_1$ be the discovered loop with largest radius. The image of the loops in the to-be-explored domain under $f^{\dag,\eps}_1$ has the same law as $\CLE$ in the unit disc conditioned on the event that the origin is in the gasket. Thus we can repeat the same procedure, define $f^{\dag,\eps}_2$, $\gamma^{\dag,\eps}_2$ etc. For $k\ge 1$, define
\[\Phi^{\dag,\eps}_k=f^{\dag,\eps}_k\circ\cdots\circ f^{\dag,\eps}_1.\]
We also need to keep track of the point $z$: let $z_k=\Phi^{\dag,\eps}_k(z)$, and let $K$ be the largest $k$ such that $z\in (\Phi^{\dag,\eps}_k)^{-1}(\D)$. Define another auxiliary stopping time $K'\le K$ as the first step $k$ at which either $|z_k|\ge 1-\delta$ or $k=K$. If $K'<K$, this means that the point $z$ is conformally far from the origin and is likely to be cut off in the discrete exploration.

We first address the case that $z$ is discovered at step $K+1$. Note that $\Phi^{\dag,\eps}_K$ will converge in distribution towards some conformal map $\Psi^{\dag}_S$ (with similar explanation as for Proposition \ref{res::caratheodory_convergence}) obtained from the Poisson point process $(\gamma^{\dag}_t,t\ge 0)$. This implies that $L^{\dag}(z)$ has the same law as $(\Psi^{\dag}_S)^{-1}(\gamma^{\dag}_S)$ as we expected. Now we will deal with the case that $z$ is cut off from the origin: we stop the discrete exploration at step $K'-1$. At step $K'$, instead of discovering the loops intersecting the ball of radius $\eps$, we discover the loops intersecting the circle centered at $z_{K'-1}$ with radius $\sqrt{\delta}$. After this step, we continue the discrete exploration (of size $\eps$) by targeting at the image of the point $z$, in the same way for the discrete exploration of simple $\CLE$ targeted at the image of the point $z$. We will discover the point $z$ at some step. Let $\eps$ and $\delta$ go to zero in proper way, we can also prove the conclusion for $L^{\dag}(z)$ in this case.

We also need to show the conclusion for the joint law of $(L^{\dag}(z_1),...,L^{\dag}(z_k))$ where $z_1,...,z_k\in\DO$. The argument is almost the same as above and we leave it to the interested reader.

\section{CLE in the punctured plane}
\label{sec::cle_punctured_plane}
%\subsection{Existence and properties of CLE in punctured plane}
In this section we discuss $\CLE$ in the punctured plane. The following lemma is analogous to Lemma \ref{lem::annuluscles_coupling}, and we delay its proof to the end of this section.
\begin{lemma}\label{lem::annuluscles_coupling_wholeplane}
There exists a universal constant $C<\infty$ such that the following is true. Let $\delta\in (0,1), 0<r'<r<\delta^2$, and let $D \subset \A(\delta,1/\delta)$.
Suppose $\Gamma_r$ (resp. $\Gamma_{r'}$) is a $\CLE$ in the annulus $\A(r,1/r)$ (resp. $\A(r',1/r')$) and $D_r^*$ (resp. $D^*_{r'}$) is the set obtained by removing from $D$ all loops (and their interiors) of $\Gamma_r$ (resp. $\Gamma_{r'}$) that are not totally contained in $D$.
Then there exists a coupling between $\Gamma_r$ and $\Gamma_{r'}$ such that the probability of the event $\{D^*_r=D^*_{r'}\}$ is at least
\[1-C\frac{\log(1/\delta)}{\log(1/r)}.\]
Furthermore, on the event $\{D^*_r=D^*_{r'}\}$, the collection of loops of $\Gamma_r$ restricted to $D_r^*$ is the same as the collection of loops of $\Gamma_{r'}$ restricted to $D_{r'}^*$.
\end{lemma}

\begin{theorem}\label{thm::conditionedcle_wholeplane_construction}
There exists a unique measure on collections of disjoint simple loops in the punctured plane, which we call \textbf{$\CLE$ in the punctured plane}, or \textbf{$\CLE$ in $\C$ conditioned on the event that both the origin and infinity are in the gasket}, to which $\CLE$ in the annulus $\A(r,1/r)$ converges in the following sense. There exists a universal constant $C<\infty$ such that for any $\delta>0$ and any subset $D\subset A(\delta,1/\delta)$, if $\Gamma^{\dag}$ is a $\CLE$ in the punctured plane and $\Gamma_r$ is a $\CLE$ in the annulus $\A(r,1/r)$, and $D^{\dag,*}$ (resp. $D^*_r$) is the set obtained by removing from $D$ all loops of $\Gamma^{\dag}$ (resp. $\Gamma_r$) that are not totally contained in $D$, then $\Gamma^{\dag}$ and $\Gamma_r$ can be coupled so that the probability of the event $\{D^{\dag,*}=D^*_r\}$ is at least
\[1-C\frac{\log(1/\delta)}{\log(1/r)}.\]
Furthermore, on the event $\{D^{\dag,*}=D^*_r\}$, the collection of loops of $\Gamma^{\dag}$ restricted to $D^{\dag,*}$ is the same as the collection of loops of $\Gamma_r$ restricted to $D_r^*$.
\end{theorem}

It is clear that $\CLE$ in the punctured plane can also be viewed as the limit of $\CLE$ in $R\D$ conditioned on the event that the origin is in the gasket as $R\to\infty$ or the limit of $\CLE$ in $\C\setminus r\D$ conditioned on the event that infinity is in the gasket as $r\to 0$.

\begin{proposition}\label{prop::conditionedcle_wholeplane_confinv}
$\CLE$ in the punctured plane satisfies the conformal invariance:
\begin{enumerate}
\item [(1)] $\CLE$ in the punctured plane is invariant under the conformal map: $z\mapsto \lambda z$, for any $\lambda\in\C$.
\item [(2)] $\CLE$ in the punctured plane is invariant under the conformal map: $z\mapsto 1/z$.
\end{enumerate}
\end{proposition}

Note that, Lemma \ref{lem::annuluscles_coupling_wholeplane} and Theorem \ref{thm::conditionedcle_wholeplane_construction} are the counterparts of Lemma \ref{lem::annuluscles_coupling} and Theorem \ref{thm::conditionedcle_construction}. The proof of Theorem \ref{thm::conditionedcle_wholeplane_construction} from Lemma \ref{lem::annuluscles_coupling_wholeplane} is almost the same as the proof of Theorem \ref{thm::conditionedcle_construction} from Lemma \ref{lem::annuluscles_coupling}. The conformal invariance of $\CLE$ in the punctured plane in Proposition \ref{prop::conditionedcle_wholeplane_confinv} is then a direct consequence of the construction in Theorem \ref{thm::conditionedcle_wholeplane_construction}. The rest of this subsection is devoted to the proof of Lemma \ref{lem::annuluscles_coupling_wholeplane}, since the proof of Lemma \ref{lem::annuluscles_coupling} does not work directly here, we need some extra effort to complete the proof of Lemma \ref{lem::annuluscles_coupling_wholeplane}.
\begin{proof}[Proof of Lemma \ref{lem::annuluscles_coupling_wholeplane}]
We first introduce a quantity $q(r)$ for $r>0$ small: Let $\LL$ be a Brownian loop soup in $\A_r$, define $E(\LL)$ as the event that there is no cluster of $\LL$ disconnecting $C_r$ from $C_1$. Suppose $\gamma$ is a continuous path in $\A_r$ connecting $C_r$ to $C_1$. Define $E_{\gamma}(\LL)$ to be the event that there is no cluster of $\LL\cup\{\gamma\}$ disconnecting $C_r$ from $C_1$. See Figure \ref{fig::conditionedcle_wholeplane_q}. Clearly, $E_{\gamma}(\LL)\subset E(\LL)$. Define
\[q(r)=\sup_{\gamma}\PP[E_{\gamma}(\LL)\cond E(\LL)],\]
where the $\sup$ is taken over all possible continuous paths $\gamma$ in $\A_r$ that connect $C_r$ to $C_1$. We can see that $q(r)\to 0$ as $r$ goes to zero.

\begin{figure}[ht!]
\begin{center}
\includegraphics[width=0.23\textwidth]{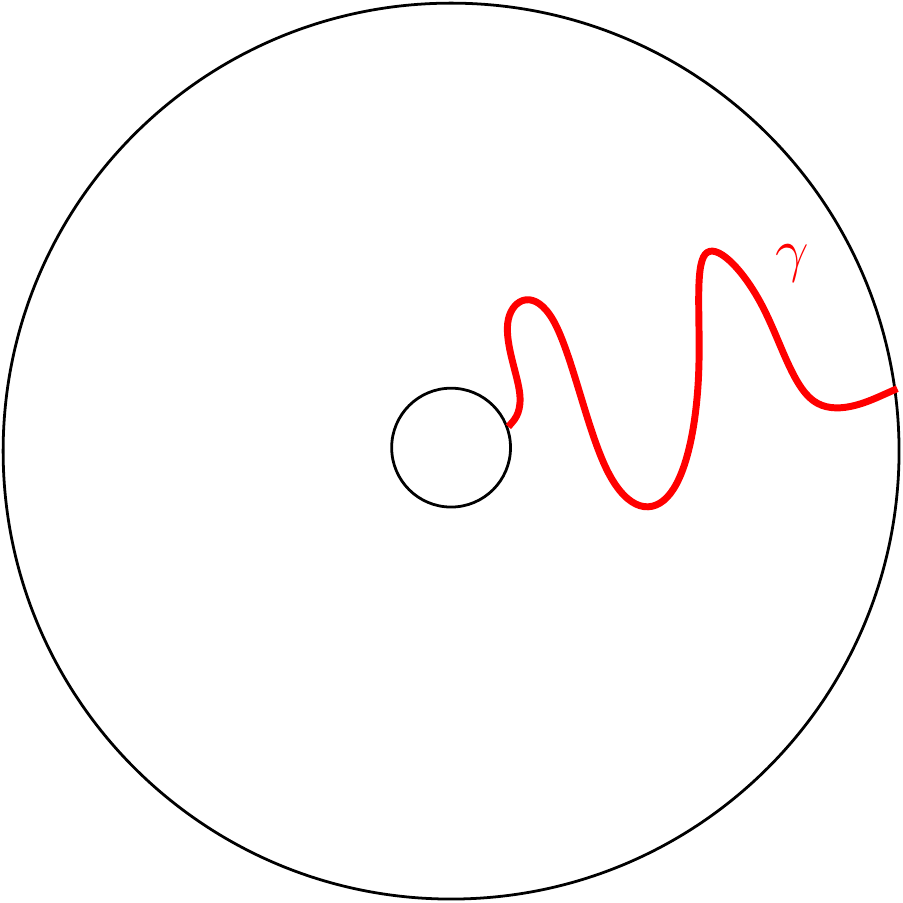}
\includegraphics[width=0.23\textwidth]{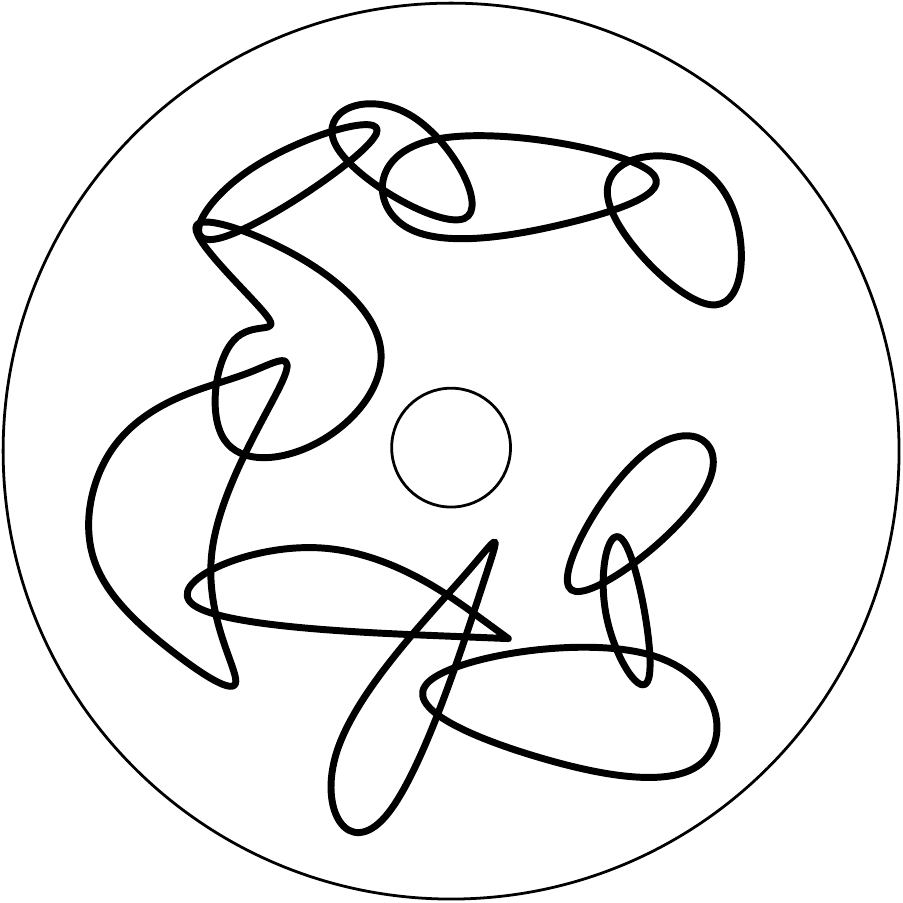}
\includegraphics[width=0.23\textwidth]{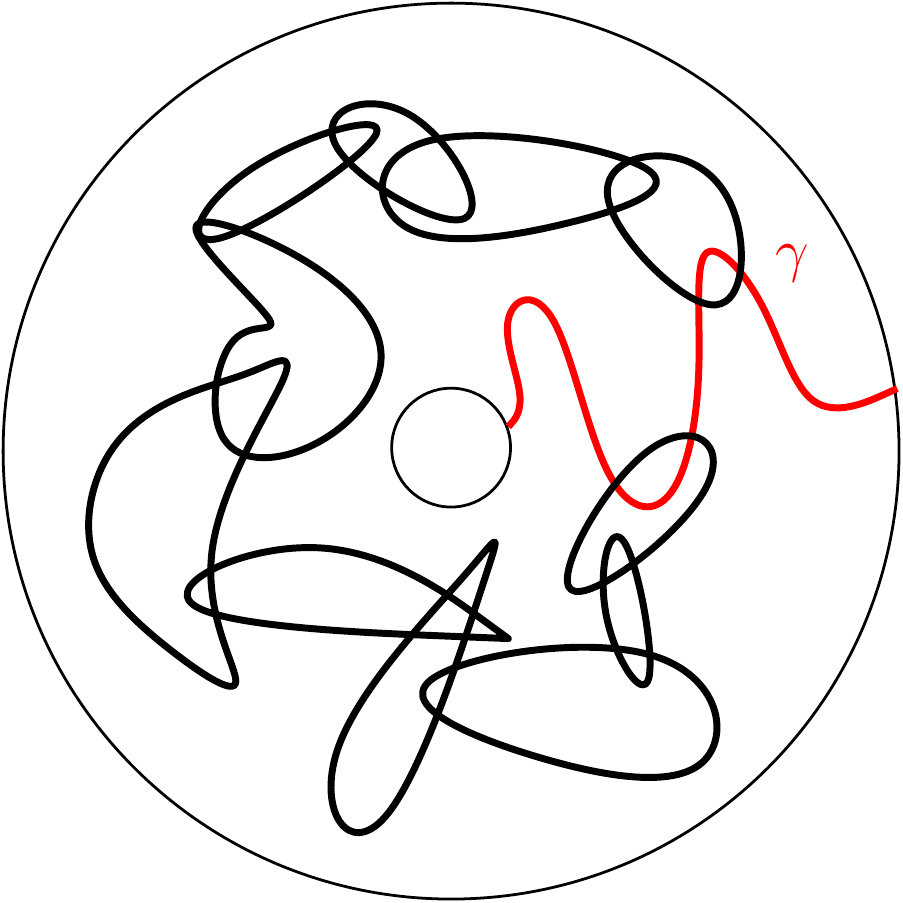}
\end{center}
\caption{\label{fig::conditionedcle_wholeplane_q} The first panel indicates a continuous path $\gamma$ in the annulus connecting the two pieces of the boundary. The second panel indicates a Brownian loop soup in the annulus and $E(\LL)$ holds. The third panel indicates that $E_{\gamma}(\LL)$ does not hold even though $E(\LL)$ holds.}
\end{figure}

Take $r,r'>0$ small. Let $\LL$ be a Brownian loop soup in $\A_{rr'}$. Suppose $\LL_1$ (resp. $\LL_2$) is the collection of loops of $\LL$ that are contained in $\A(rr',r)$ (resp. $\A_r$).
Let $\gamma$ be any continuous path in $\A_{rr'}$ connecting $C_{rr'}$ to $C_1$. Suppose $\gamma_1$ (resp. $\gamma_2$) is part of $\gamma$ that is a continuous path in $\A(rr',r)$ (resp. $\A_r$) connecting $C_{rr'}$ to $C_r$ (resp. connecting $C_r$ to $C_1$). Then we have that
\begin{eqnarray*}
\lefteqn{\PP[E_{\gamma}(\LL)|E(\LL)]}\\
&=&\PP[E_{\gamma}(\LL)]/p(rr')\\
&\le& \PP[E_{\gamma_1}(\LL_1),E_{\gamma_2}(\LL_2)]/p(rr')\\
&=&\PP[E_{\gamma_1}(\LL_1)\cond E(\LL_1)]\times\PP[E_{\gamma_2}(\LL_2)\cond E(\LL_2)]\times p(r')p(r)/p(rr')\\
&\lesssim& q(r)q(r').
\end{eqnarray*}
Thus, there exists universal constant $C$ so that
\[q(rr')\le Cq(r)q(r').\]
Together with the fact that $q(r)\to 0$ as $r$ goes to zero, we have that there exists some constant $\tilde{\alpha}>0$ such that, for $r>0$ small, $$q(r)\le r^{\tilde{\alpha}}.$$

Now we are ready to complete the proof. Suppose $\LL$ is a Brownian loop soup in $\A(r',1/r')$. Let $\LL_1$ be the collection of loops of $\LL$ that are contained in $\A(r,1/r)$. On the event $E(\LL)$, let $\Gamma$ (resp. $\Gamma_1$) be the collection of the outer boundaries of outermost clusters of $\LL$ (resp. $\LL_1$). Let $D^*$ (resp. $D_1^*$) be the set obtained by removing from $D$ all loops of $\Gamma$ (resp. $\Gamma_1$) that are not totally contained in $D$.
Note that, if $\{D^*\neq D_1^*\}$, there must exists a loop in $\LL$ intersecting both $C_r$ and $C_{\delta}$ or intersecting both $C_{1/r}$ and $C_{1/\delta}$.
Define $S(\LL,C_r,C_{\delta})$ to be the event that there exists a loop of $\LL$ that intersects both $C_r$ and $C_{\delta}$. Then we have that
\begin{eqnarray*}
\lefteqn{\PP[D^*\neq D_1^*, E(\LL)]/p\left((r')^2\right)}\\
&\le &2\PP[S(\LL,C_r,C_{\delta}), E(\LL)]/p\left((r')^2\right).
\end{eqnarray*}

We divide the loops in $\LL$ into independent collections: Let $\LL_2$ be the loops in $\LL$ that are contained in $\A(r',r)$, $\LL_3$ be the loops in $\LL$ that are contained in $\A(r,\delta)$, $\LL_4$ be the loops in $\LL$ that are contained in $\A(\delta,1/r')$, and $\LL_5$ be the collection of loops in $\LL$ that intersect both $C_r$ and $C_{\delta}$. Clearly, $\LL_2,\LL_3,\LL_4,\LL_5$ are independent and the event $S(\LL,C_r,C_{\delta})$ is the same as $\{\LL_5\neq\emptyset\}$. Define $E_2$ (resp. $E_3$, $E_4$) to be the event that there is no cluster of $\LL_2$ (resp. $\LL_3$, $\LL_4$) disconnecting $C_{r'}$ from $C_r$ (resp. disconnecting $C_r$ from $C_{\delta}$, disconnecting $C_{\delta}$ from $C_{1/r'}$).
Then $E_2$, $E_3$, $E_4$ are independent, and their probabilities are $p(r'/r)$, $p(r/\delta)$, $p(\delta r')$ respectively. Thus
\begin{eqnarray*}
\lefteqn{\PP[S(\LL,C_r,C_{\delta}),E(\LL)]/p\left((r')^2\right)}\\
&=&\PP[\LL_5\neq\emptyset,E(\LL),E_2,E_3,E_4]/p\left((r')^2\right)\\
&=&\PP[\LL_5\neq\emptyset,E(\LL)\cond E_2,E_3,E_4]\times p(r'/r)p(r/\delta)p(\delta r')/p\left((r')^2\right)\\
&\lesssim & \PP[\LL_5\neq\emptyset,E(\LL)\cond E_2,E_3,E_4]\\
&\le&\PP[\LL_5\neq\emptyset, E(\LL_3\cup\LL_5)\cond E_2,E_3, E_4]\\
&=&\PP[\LL_5\neq\emptyset, E(\LL_3\cup\LL_5)\cond E_3]\\
&\le & q(r/\delta),\end{eqnarray*}
where $E(\LL_3\cup\LL_5)$ is the event that there is no cluster of $\LL_3\cup\LL_5$ disconnecting $C_r$ from $C_{\delta}$. This implies the conclusion.
\end{proof}

\bibliographystyle{alpha}
\bibliography{cle_doubly_connected}
\bigbreak
\noindent Department of Mathematics\\
\noindent Massachusetts Institute of Technology\\
\noindent Cambridge, MA, USA\\
\noindent sheffield@math.mit.edu\\
\noindent sswatson@mit.edu\\
\noindent hao.wu.proba@gmail.com

\end{document}